\documentclass[reqno]{amsart}

\usepackage{color}
\usepackage[foot]{amsaddr}
\usepackage{booktabs}
\usepackage{bigstrut}
\usepackage{amsmath,amsthm,bm}
\usepackage{amsfonts}
\usepackage{algorithmic}
\usepackage{algorithm}
\usepackage{multirow}
\usepackage{overpic}
  \usepackage{paralist}
  \usepackage{graphics} 
  \usepackage{epsfig} 
\usepackage{graphicx}
\usepackage{epstopdf} 

\usepackage{float}

\usepackage{tikz}
\usetikzlibrary{calc}

\usepackage{url}

\newtheorem{theorem}{Theorem}[section]

\newtheorem{lemma}[theorem]{Lemma}
\newtheorem{proposition}[theorem]{Proposition}

\newtheorem{definition}[theorem]{Definition}
\theoremstyle{definition}
\newtheorem{remark}[theorem]{Remark}

\DeclareMathOperator*{\esssup}{ess\,sup}

\graphicspath{{images/}}

\begin{document}

\title[Hypergraph $p$-Laplacian regularization  on point clouds]
{Hypergraph $p$-Laplacian regularization on point clouds for data interpolation}

\author{Kehan Shi $^{\ast}$, Martin Burger $^{\dagger,\ddagger}$}
\address{$^{\ast}$Department of Mathematics, China Jiliang University,
Hangzhou 310018, China}
\address{$^{\dagger}$Computational Imaging Group and Helmholtz Imaging, Deutsches Elektronen-Synchrotron DESY, 22607 Hamburg, Germany}
\address{$^{\ddagger}$Fachbereich Mathematik, Universit\"{a}t Hamburg, 20146 Hamburg, Germany}


\subjclass{}
 \keywords{Hypergraph, $p$-Laplacian, continuum limit, stochastic primal-dual algorithm, data interpolation}


\begin{abstract}
  As a generalization of graphs, hypergraphs are widely used to model higher-order relations in data.
  This paper explores the benefit of the hypergraph structure for the interpolation of point cloud data that contain no explicit structural information.
  We define the $\varepsilon_n$-ball hypergraph and the $k_n$-nearest neighbor hypergraph on a point cloud and study the $p$-Laplacian regularization on the hypergraphs.
  We prove the variational consistency between the hypergraph $p$-Laplacian regularization and the continuum $p$-Laplacian regularization in a semisupervised setting when the number of points $n$ goes to infinity while the number of labeled points remains fixed.
  A key improvement compared to the graph case is that the results rely on weaker assumptions on the upper bound of $\varepsilon_n$ and $k_n$.
  To solve the convex but non-differentiable large-scale optimization
  problem, we utilize the stochastic primal-dual hybrid gradient algorithm.
  Numerical experiments on data interpolation verify that the hypergraph $p$-Laplacian regularization outperforms the graph $p$-Laplacian regularization in preventing the development of spikes at the labeled points.
\end{abstract}

\maketitle

\section{Introduction}
A hypergraph $H$ is a triplet with the form $H=(V, E, W)$, where $V=\{x_i\}_{i=1}^n$ is the vertex set, $E=\{e_k\}_{k=1}^m$ denotes the set of hyperedges, and $W=\{w_k\}_{k=1}^m$ assigns weight $w_k>0$ for hyperedge $e_k$.
It is a generalization of a graph in the sense that each hyperedge $e_k\subset V$ contains an arbitrary number of vertices.
Due to the applicability of modeling higher-order relations in data, hypergraphs have been extensively used in image processing \cite{zhang2020hypergraph,zeng2023multi}, bioinformatics \cite{klamt2009hypergraphs,patro2013predicting},  social networks \cite{antelmi2021social,fazeny2023hypergraph}, etc.



This paper focuses on the $p$-Laplacian regularization on hypergraph $H$,
\begin{equation}\label{eq:1.1}
  \sum_{k=1}^m w_k \max_{x_i,x_j\in e_k}|u(x_i)-u(x_j)|^p,\quad p> 1,
\end{equation}
which was first proposed in \cite{hein2013total} for clustering and semi-supervised learning (SSL).
The authors defined the hypergraph total variation (i.e., energy \eqref{eq:1.1} with $p=1$) in analogy to the graph case \cite{bach2013learning}, i.e., it is the Lov\'{a}sz extension of the hypergraph cut. Then a power $p>1$ was introduced to avoid spiky solutions in SSL.
Despite the successful applications of energy \eqref{eq:1.1} and its variants \cite{zhang2017re,li2018submodular,yadati2019hypergcn},
some of its mathematical properties are not clear.
One of the fundamental questions is whether it is a discrete approximation of
a continuum energy related to the $p$-Laplacian regularization
\begin{equation*}
  \int_{\Omega}|\nabla u|^pdx,\quad p>1.
\end{equation*}
Since the latter is well-studied, a positive answer shall provide us with more insight into the energy \eqref{eq:1.1}.

The difficulty comes from the complicated structure of the hypergraph.
In hypergraph-based learning, the first step is always to construct a hypergraph from the given data.
This is neither trivial nor unique \cite{gao2020hypergraph}.
Typically, we construct the hyperedge using structural information, such as the attribute \cite{huang2015learning} or the network \cite{fang2014topic}, that is explicitly contained in the data. This means that the formulation of energy \eqref{eq:1.1} depends on the data, thus preventing us from establishing general results for it.

In this paper, we define the $\varepsilon_n$-ball hypergraph $H_{n,\varepsilon_n}$ from a given point cloud $\Omega_n=\{x_i\}_{i=1}^n$ and investigate the hypergraph $p$-Laplacian regularization on it in a semisupervised setting.
It is assumed that the first $N$ points $\mathcal{O}:=\{x_i\}_{i=1}^N$ are labeled with labels $\{y_i\}_{i=1}^N\subset\mathbb{R}$ and the remaining $n-N$ points are unlabeled and drawn from a probability measure $\mu$ supported in a bounded set $\Omega\subset\mathbb{R}^d$.
The given point cloud $\Omega_n$ forms the vertices of hypergraph $H_{n,\varepsilon_n}$ and
each hyperedge is a subset of $\Omega_n$ consisting of a vertex and its $\varepsilon_n$-ball neighbors, i.e.,
\begin{equation*}
  V=\Omega_n,\quad E=\{e_k:=B(x_k,\varepsilon_n), x_k\in \Omega_n\}_{k=1}^n, ~~\varepsilon_n>0.
\end{equation*}
The homogeneous weight $W=\{w_k:= 1\}_{k=1}^n$ is assigned.
Then semi-supervised learning with $p$-Laplacian regularization \eqref{eq:1.1} on the hypergraph $H_{n,\varepsilon_n}$ is formulated as the problem of finding an estimator $u:\Omega_n\rightarrow\mathbb{R}$ by minimizing
\begin{equation}\label{eq:1.2}
  \mathcal{E}_{n,\varepsilon_n}(u)=\frac{1}{n\varepsilon_n^{p}}\sum_{k=1}^{n}\max_{x_i,x_j\in B(x_k,\varepsilon_n)}|u(x_i)-u(x_j)|^p,\quad p>1,
\end{equation}
under constraint $u(x_i)=y_i$ for $x_i\in\mathcal{O}$.
Here a scaling parameter $\frac{1}{n\varepsilon_n^{p}}$ is introduced to ensure that the energy is well-defined when $n\rightarrow\infty$.
The method makes use of the smoothness assumption.
Namely, vertices tend to share a label if they are  close to each other such that they are in the same hyperedge.






We are interested in the asymptotic behavior of $\mathcal{E}_{n,\varepsilon_n}$ when the number of data points $n$ goes to infinity while the number of labeled points $N$ remains fixed.
This corresponds to the learning problem that assigns labels to a large number of unlabeled data through a given few labeled data.
It has tremendous practical value in applications where
acquiring labeled data is expensive and time-consuming, while acquiring unlabeled data is relatively easy.

The continuum limit of $\mathcal{E}_{n,\varepsilon_n}$ as $n\rightarrow\infty$ reads
\begin{align}\label{eq:1.3}
  \mathcal{E}(u)=\mathcal{E}(u;\rho)=
  \begin{cases}
    2^p\int_{\Omega}|\nabla u|^p \rho dx, & \mbox{if } u\in W^{1,p}(\Omega), \\
    +\infty, & \mbox{otherwise},
  \end{cases}
\end{align}
where $\rho$ is the density of the probability measure $\mu$ with respect to the Lebesgue measure.
To establish the connection between the discrete energy $\mathcal{E}_{n,\varepsilon_n}$ and the continuum energy $\mathcal{E}$, we utilize the
method developed in \cite{garcia2016continuum,slepcev2019analysis}, where continuum limits of variational models on graphs were studied by tools of optimal transportation and $TL^p$ topology.
More precisely, we prove that $\mathcal{E}_{n,\varepsilon_n}$ $\Gamma$-converges to $\mathcal{E}$ in the $TL^p$ topology as $n\rightarrow\infty$ if the connection radius $\varepsilon_n$ satisfies
\begin{equation}\label{eq:1.4}
  \delta_n\ll\varepsilon_n\ll 1,
\end{equation}
where
  \begin{align}\label{delta}
    \delta_n=
      \begin{cases}
        \sqrt{\frac{\ln\ln(n)}{n}}, & \mbox{if } d=1, \\
        \frac{(\ln n)^{3/4}}{\sqrt{n}}, & \mbox{if } d=2, \\
        \frac{(\ln n)^{1/d}}{n^{1/d}}, & \mbox{if  } d\geq 3,
      \end{cases}
    \end{align}
is a constant depending on $n$ and $d$.

In the case of the constrained energy
\begin{align}\label{eq:1.5}
  \mathcal{E}^{con}_{n,\varepsilon_n}(u)=
  \begin{cases}
    \mathcal{E}_{n,\varepsilon_n}(u), & \mbox{if } u(x_i)=y_i \mbox{ for } x_i\in\mathcal{O}, \\
    +\infty, & \mbox{otherwise},
  \end{cases}
\end{align}
where the training set $\{(x_i,y_i)\}_{i=1}^N$ is taken into account, the corresponding continuum limit becomes
\begin{align}\label{eq:1.6}
  \mathcal{E}^{con}(u)=\mathcal{E}^{con}(u;\rho)=
  \begin{cases}
    \mathcal{E}(u), & \mbox{if } u\in W^{1,p}(\Omega) \mbox{ and } u(x_i)=y_i \mbox{ for } x_i\in\mathcal{O}, \\
    +\infty, & \mbox{otherwise}.
  \end{cases}
\end{align}
If $p>d$, Sobolev's embedding theorem \cite{adams2003sobolev} implies that the minimizer of $\mathcal{E}^{con}$ is H\"older continuous and the constraint is well-defined.
Under assumptions \eqref{eq:1.4} and $p>d$, we also have that
$\mathcal{E}^{con}_{n,\varepsilon_n}$ $\Gamma$-converges to $\mathcal{E}^{con}$ in the $TL^p$ topology as $n\rightarrow\infty$.
As a consequence,
the minimizer of $\mathcal{E}^{con}_{n,\varepsilon_n}$ converges to the minimizer of $\mathcal{E}^{con}$ in the $TL^p$ topology as $n\rightarrow\infty$.
In other words,
$\mathcal{E}^{con}_{n,\varepsilon_n}$ is a discrete approximation of the continuum functional $\mathcal{E}^{con}$ on the point cloud.
If $p\leq d$,
$\mathcal{E}^{con}=\mathcal{E}$ and
the minimizer of $\mathcal{E}^{con}_{n,\varepsilon_n}$ converges to a minimizer of $\mathcal{E}$ instead, which is a constant. This means that
$\mathcal{E}^{con}_{n,\varepsilon_n}$ is degenerate and develops spiky solutions for large $n$.

A similar result holds for the constrained $p$-Laplacian regularization on the random geometric graph under an additional assumption on the upper bound of the connection radius $\varepsilon_n$, i.e.,
\begin{equation*}
  \delta_n\ll\varepsilon_n\ll \left(\frac{1}{n}\right)^{\frac{1}{p}}.
\end{equation*}
See \cite{slepcev2019analysis} and Section \ref{se:1.1} for more details. This reveals that the hypergraph structure is beneficial in a semisupervised setting even for point cloud data that contain no explicit structural information.

The cardinality for the hyperedge of the $\varepsilon_n$-ball hypergraph generally varies significantly.
This causes high computational costs in applications.
A more feasible alternative is to consider the $p$-Laplacian regularization on the $k_n$-nearest neighbor ($k_n$-NN) hypergraph $H_{n,k_n}$, which reads
\begin{equation}\label{eq:1.7}
  \mathcal{F}_{n,k_n}(u)=\frac{1}{n\bar{\varepsilon}_n^{p}}\sum_{k=1}^{n}
  \max_{x_i,x_j\stackrel{k_n}{\sim} x_k}|u(x_i)-u(x_j)|^p,\quad p>1.
\end{equation}
Here
\begin{equation}\label{eq:1.8}
  \bar{\varepsilon}_n=\left(\frac{1}{\alpha_d}\frac{k_n}{n}\right)^{1/d},
\end{equation}
$\alpha_d$ denotes the volume of the $d$-dimensional unit ball,
and  $x_i\stackrel{k_n}{\sim} x_k$ denotes that vertex $x_i$ is among the $k_n$ nearest points to vertex $x_k$.

The motivation for the definition of $\bar{\varepsilon}_n$ is as follows.
Let $\mu_n=\frac{1}{n}\sum_{i=1}^n\delta_{x_i}$ be the empirical measure of $\Omega_n$.
If $\rho\equiv 1$ and $n$ is large, then for any $x_k\in\Omega_n$ and a constant connection radius $\varepsilon_n>0$,
\begin{equation*}
  \mu_n(B(x_k,\varepsilon_n))\approx \mu(B(x_k,\varepsilon_n))=\alpha_d\varepsilon_n^d.
\end{equation*}
This is what we are doing on the $\varepsilon_n$-ball hypergraph,
while on the $k_n$-NN hypergraph, the connection radius $\varepsilon_n$ is now a function of $x_k$ such that
\begin{equation*}
  \mu_n(B(x_k,\varepsilon_n(x_k)))=\frac{k_n}{n}.
\end{equation*}
For sufficiently large $n$,
\begin{align*}
   \mu(B(x_{k_1},\varepsilon_n(x_{k_1})))
   &\approx \mu_n(B(x_{k_1},\varepsilon_n(x_{k_1})))=\frac{k_n}{n}\\
  &=\mu_n(B(x_{k_2},\varepsilon_n(x_{k_2}))) \approx \mu(B(x_{k_2},\varepsilon_n(x_{k_2}))),
\end{align*}
which means that $\varepsilon_n(x_k)$ is approximately a constant function (denoted by $\bar{\varepsilon}_n$).
Then we have
\begin{equation*}
  \alpha_d\bar{\varepsilon}_n^d=\mu(B(x_k,\bar{\varepsilon}_n))\approx\mu_n(B(x_k,\bar{\varepsilon}_n))\approx\mu_n(B(x_k,\varepsilon_n(x_k)))=\frac{k_n}{n},
\end{equation*}
from which we obtain \eqref{eq:1.8}.

Under the assumption
\begin{equation*}
  \delta_n\ll\bar{\varepsilon}_n\ll 1,
\end{equation*}
all theoretical results regarding the $\varepsilon_n$-ball hypergraph $p$-Laplacian regularization $\mathcal{E}_{n,\varepsilon_n}$ and $\mathcal{E}^{con}_{n,\varepsilon_n}$ also hold for $\mathcal{F}_{n,k_n}$
and $\mathcal{F}^{con}_{n,k_n}$, where
\begin{align}\label{eq:1.9}
  \mathcal{F}^{con}_{n,k_n}(u)=
  \begin{cases}
    \mathcal{F}_{n,k_n}(u), & \mbox{if } u(x_i)=y_i \mbox{ for } x_i\in\mathcal{O}, \\
    +\infty, & \mbox{otherwise}.
  \end{cases}
\end{align}
The only difference is that the continuum limits possess a different weight. Namely, $\mathcal{F}_{n,k_n}$ and $\mathcal{F}^{con}_{n,k_n}$ ($p>d$) $\Gamma$-converge to $\mathcal{E}(~\cdot~; \rho^{1-p/d})$ and $\mathcal{E}^{con}(~\cdot~; \rho^{1-p/d})$ as $n\rightarrow\infty$ respectively.

Due to the maximum function, energy \eqref{eq:1.5} and energy \eqref{eq:1.9} are convex but non-differentiable for any $p\geq 1$.
 The primal-dual hybrid gradient (PDHG) algorithm \cite{chambolle2011first} was used in \cite{hein2013total} when $p=1, 2$.
Considering that the number of hyperedges $n$ is usually large,
the recent stochastic PDHG algorithm \cite{chambolle2018stochastic} provides us with a more efficient scheme, in which it updates a random subset of the separable dual variable in each iteration.
The resulting algorithm updates the primal variable by projecting onto the training set, and updates the dual variable by solving the proximal operator of $\|\cdot\|_1^{p/{p-1}} (p>1)$. The latter has no closed-form solution but can be solved exactly in a few iterations.
In the case $p=1$, the subproblem involving the dual variable becomes a projection onto the $L^1$ ball.

The theoretical results are partially verified by numerical experiments.
By comparing the hypergraph $p$-Laplacian regularization and the graph $p$-Laplacian regularization for data interpolation in 1D, we observe that the hypergraph structure can better suppress spiky solutions.
Experiments on higher-dimensional data interpolation, including SSL and image inpainting, are also presented to obtain the same conclusion.

At last, we mention that
energy \eqref{eq:1.2} and energy \eqref{eq:1.7} consist of $n$ terms, each of which is nothing but a $p$-th power of the objective function of the Lipschitz learning \cite{von2004distance,kyng2015algorithms} on a hyperedge.
The Lipschitz learning (see \eqref{eq:01.3} for its definition) has been proven to be effective in suppressing spikes \cite{calder2019consistency,roith2023continuum}.
This informally explains the numerical results that the hypergraph $p$-Laplacian regularization suppresses spikes better than the graph one.

This paper is organized as follows. We complete the introduction with a brief review of the graph $p$-Laplacian regularization. In Section 2, we present definitions of the $\varepsilon_n$-ball hypergraph and the $k_n$-NN hypergraph, and present mathematical tools needed for the following study.
The continuum limits of the $p$-Laplacian regularization on the $\varepsilon_n$-ball hypergraph and the $k_n$-NN hypergraph in a semisupervised setting are studied in Section 3 and Section 4, respectively.
Section 5 is devoted to the numerical algorithm, which is based on the stochastic primal-dual hybrid gradient method.
Numerical experiments are presented in Section 6.
We conclude this paper in Section 7.

\subsection{Related works on the graph $p$-Laplacian regularization}\label{se:1.1}
The graph-based method has been widely used for SSL. It constructs a weighted graph $G_n$ with vertex set $\Omega_n=\{x_i\}_{i=1}^n$ to represent the geometric structure in $\Omega_n$.
We follow the previous notation and assume that the first $N$ vertices $\mathcal{O}=\{x_i\}_{i=1}^N$ are labeled with labels $\{y_i\}_{i=1}^N$.
It is commonly assumed in the theoretical study that $G_n$ is a random geometric graph. Namely, the unlabeled vertices $\{x_i\}_{i=N+1}^n$ are drawn from a probability measure $\mu$ and an edge $e_{i,j}$ between two vertices $x_i$ and $x_j$ exists if and only if their distance is smaller than a given connection radius $\varepsilon_n$. The weight $w_{i,j}>0$ for edge $e_{i,j}$ is a nonincreasing function of the distance.
Then an estimator is obtained by minimizing an objective functional defined on $G_n$ under the constraint on the labeled subset.


The graph Laplacian regularization
\begin{equation*}
  \sum_{i,j=1}^nw_{i,j}|u(x_i)-u(x_j)|^2,\quad \textrm{subject to } u(x_i)=y_i \textrm{ for } x_i\in\mathcal{O},
\end{equation*}
 was first proposed in \cite{zhu2003semi} for SSL and is one of the most well-known approaches.
It has attracted a lot of attention \cite{zhou2005learning,ando2006learning} and has applications in clustering \cite{von2007tutorial}, manifold ranking  \cite{yang2013saliency}, image processing \cite{gilboa2007nonlocal}, etc.
Still, one of its drawbacks can not be neglected.
It is degenerate in the sense that the minimizer becomes noninformative and develops spikes at the labeled points $x_i\in\mathcal{O}$ when the labeling rate $\frac{|\mathcal{O}|}{|\Omega_n|}=\frac{N}{n}$ is low \cite{nadler2009statistical}.
Variants have been proposed as a remedy:
In \cite{shi2017weighted,calder2020properly}, the authors proposed to assign more weights for edges adjacent to the labeled points.
Instead of imposing the labels as boundary conditions, the authors of \cite{calder2020poisson} constructed a source term with the labels and proposed the Poisson learning.

In \cite{el2016asymptotic}, the authors suggested to use the $p$-Laplacian regularization \cite{zhou2005regularization,elmoataz2015p} with $p>d$, which has the following form
\begin{equation}\label{eq:GpL}
  E_{n,\varepsilon_n}^{con}(u)=
  \begin{cases}
  \frac{1}{n^2\varepsilon_n^p}\sum_{i,j=1}^nw_{i,j}|u(x_i)-u(x_j)|^p, \quad &\mbox{if } u(x_i)=y_i \mbox{ for }x_i\in\mathcal{O},\\
  +\infty,\quad &\mbox{otherwise.}
  \end{cases}
\end{equation}
A rigorous study \cite{slepcev2019analysis} indicates that if
$\delta_n\ll\varepsilon_n\ll \left(\frac{1}{n}\right)^{\frac{1}{p}}$,
energy $E_{n,\varepsilon_n}^{con}$ is variational consistent in the $TL^p$ topology with the continuum energy
\begin{align*}
  E^{con}(u)=
  \begin{cases}
    \sigma_p\int_\Omega|\nabla u(x)|^p\rho^2(x)dx, &~ \textrm{if } u\in W^{1,p}(\Omega) \mbox{ and } u(x_i)=y_i \textrm{ for } x_i\in\mathcal{O},\\
    +\infty, &~ \textrm{otherwise},
  \end{cases}
\end{align*}
where $\sigma_p>0$ is a constant.
Furthermore, the minimizer of $E_{n,\varepsilon_n}^{con}$ converges to the minimizer of $E^{con}$ in the $TL^p$ topology
as $n\rightarrow \infty$.
The assumption $\delta_n\ll\varepsilon_n\ll \left(\frac{1}{n}\right)^{\frac{1}{p}}$ and the definition of $\delta_n$ (see \eqref{delta}) imply that $p>d$.
Consequently,
 by Sobolev's embedding theorem, the minimizer of $E^{con}$ is H\"{o}lder continuous and the constraint in $E^{con}$ is well-defined.
If $p\leq d$ or $n\varepsilon_n^p\rightarrow \infty$, the continuum limiting energy of $E_{n,\varepsilon_n}^{con}$ is similar to $E^{con}$ but without the constraint on the labeled subset, i.e.,
\begin{align*}
  E(u)=
  \begin{cases}
    \sigma_p\int_\Omega|\nabla u(x)|^p\rho^2(x)dx, &~ \textrm{if } u\in W^{1,p}(\Omega),\\
    +\infty, &~ \textrm{otherwise}.
  \end{cases}
\end{align*}
This indicates that $E_{n,\varepsilon_n}^{con}$ forgets the labels for large $n$ and explains the occurrence of spikes in the graph-based SSL.

The minimizer of energy $E_{n,\varepsilon_n}^{con}$ gains more smoothness as $p$ increases.
Recently, some authors studied the case $p\rightarrow+\infty$, which leads to the Lipschitz learning for SSL \cite{von2004distance,kyng2015algorithms},
\begin{equation}\label{eq:01.3}
  \frac{1}{\varepsilon_n}\max_{x_i,x_j\in\Omega_n}w_{i,j}|u(x_i)-u(x_j)|,\quad \textrm{subject to } u(x_i)=y_i \textrm{ for } x_i\in\mathcal{O}.
\end{equation}
As expected, the minimizer of the continuum limiting energy
\begin{align*}
  E_{\infty}(u)=
  \begin{cases}
    \sigma_\infty\esssup_{x\in\Omega}|\nabla u(x)|, &~~ \textrm{if } u\in W^{1,\infty}(\Omega)  \mbox{ and } u(x_i)=y_i \textrm{ for } x_i\in\mathcal{O},\\
    \infty, &~~ \textrm{otherwise},
  \end{cases}
\end{align*}
attains the labels continuously \cite{calder2019consistency,roith2023continuum}.
However, \eqref{eq:01.3} is less attractive for SSL since it forgets the distribution of the unlabeled points as $n\rightarrow\infty$ \cite{el2016asymptotic}.
For the related $\infty$-Laplacian equation, the author of \cite{calder2019consistency} proved that the self-turning weights allow to remember the distribution.

Finally, we mention that the $k_n$-nearest neighbor ($k_n$-NN) graph has a more practical value than the random geometric graph due to its sparsity.
In \cite{garcia2019variational}, the author studied the continuum limit of the total variation on the $k_n$-NN graph for clustering.
The method is also valid for the graph $p$-Laplacian regularization with $p>1$ on the $k_n$-NN graph.
The basic idea is to replace the connection radius $\varepsilon_n$ in $E_{n,\varepsilon_n}^{con}$ by $\bar{\varepsilon}_n$ defined in \eqref{eq:1.8}.

\section{Preliminaries}

\subsection{Settings}
Let
  $$\Omega_n=\{x_1, \cdots,x_N,x_{N+1},\cdots, x_n\}$$
be a point cloud in the bounded domain $\Omega\subset\mathbb{R}^d$ ($d\geq 1$).
We are given a training set $\{(x_i, y_i), i=1,\cdots,N\}$ with $N$ labeled points $\mathcal{O}=\{x_i\}_{i=1}^N$, where each $y_i\in \{1,\cdots,L\}$ for an integer $L\geq 2$ is the label of point $x_i\in\mathcal{O}$. The task of semi-supervised learning (SSL) is to assign labels from the label set $\{1,\cdots,L\}$ to the remaining points $\{x_{N+1},\cdots,x_n\}$.

In this paper, we let the number of the labeled points $N$ be fixed and study the asymptotic behavior of variational models defined on $\Omega_n$ when $n$ goes to infinity.
It is assumed that the unlabeled points $\{x_i\}_{i=N+1}^n$ are  independently and identically distributed (i.i.d.) random samples of a probability measure $\mu$ on $\Omega$.
The density $\rho$ of $\mu$ (with respective to the Lebesgue measure) is a continuous function with positive lower and upper bounds, i.e.,
\begin{equation*}
  0< \inf_{x\in\Omega} \rho(x)\leq \sup_{x\in\Omega}\rho(x)<\infty.
\end{equation*}
Throughout this paper, $\Omega$ is a connected and bounded domain with Lipschitz boundary $\partial\Omega$.

\subsection{Hypergraphs}
We consider regularizations on hypergraphs in a semisupervised setting. The first step is to generate a hypergraph from the given data set $\Omega_n$.
Clearly, it is not unique and affects the performance for SSL.
A brief review of the generation of hypergraphs can be found in \cite{gao2020hypergraph}.
Since the data set $\Omega_n$ contains no explicit structural information, we  construct the hypergraph by the distance-based method. More precisely, each hyperedge is a subset consisting of a vertex and its neighbors.
The following two hypergraphs called the $\varepsilon_n$-ball hypergraph and the $k_n$-NN hypergraph are considered in this paper.
Such hypergraphs have been used for image classification, image segmentation, and recommender systems \cite{jin2019hypergraph,bie2025hyperg,wan2024recommendation}.

\subsubsection*{The $\varepsilon_n$-ball hypergraph}
Let $\varepsilon_n$ be a constant representing the connection radius. The $\varepsilon_n$-ball hypergraph is defined as
\begin{equation*}
  H_{n,\varepsilon_n}=(V_n, E_{n,\varepsilon_n}, W_n),
\end{equation*}
with vertices $V_n=\Omega_n=\{x_k\}_{k=1}^n$, hyperedges $E_{n,\varepsilon_n}=\{e_k^{(\varepsilon_n)}\}_{k=1}^n$, and weights $W_n=\{w_k\}_{k=1}^n$,
where
\begin{equation*}
  e_k^{(\varepsilon_n)}=\{x_j\in\Omega_n: |x_k-x_j|\leq \varepsilon_n\}
\end{equation*}
is the hyperedge corresponding to the centroid $x_k$.
In this definition, we have no preference for different hyperedges. Consequently, we let $w_k=1$ for any $k=1,2\cdots,n$.
By considering $H_{n,\varepsilon_n}$ in the $p$-Laplacian regularization \eqref{eq:1.1}, we obtain $\mathcal{E}_{n,\varepsilon_n}$ (defined in \eqref{eq:1.2}) with an additional scaling parameter.

In the following,
we assume that $\varepsilon_n\gg\delta_n$. Then, with probability one, the $\varepsilon_n$-ball graph is connected \cite{penrose2003random,garcia2016continuum}. Namely, for any $x_i,x_j\in\Omega_n$, there exist vertices $x_{k_1}=x_i, x_{k_2},\cdots, x_{k_{s-1}},
x_{k_s}=x_j$, such that $|x_{k_l}-x_{k_{l-1}}|\leq \varepsilon_n$ for any $2\leq l\leq s$.
This indicates that hyperedge $e_k^{(\varepsilon_n)}$ has cardinality $|e_k^{(\varepsilon_n)}|\geq 2$ for any $k=1,2\cdots,n$ and hypergraph $H_{n,\varepsilon_n}$ is also connected.

\subsubsection*{The $k_n$-NN hypergraph}
The $k_n$-nearest neighbor ($k_n$-NN) hypergraph is more frequently used in practice.
In contrast to the $\varepsilon_n$-ball hypergraph, the $k_n$-NN hypergraph selects a fixed $k_n\geq 2$ neighbors for each vertex to form a hyperedge.
More precisely, it is defined as
\begin{equation*}
  H_{n,k_n}=(V_n, E_{n,k_n}, W_n),
\end{equation*}
with a different hyperedge set $E_{n,k_n}=\{e_k^{(k_n)}\}_{k=1}^n$,
where
\begin{equation*}
  e_k^{(k_n)}=\{x_j\in\Omega_n: x_j\stackrel{k_n}{\sim} x_k \},
\end{equation*}
and $x_j\stackrel{k_n}{\sim} x_k$ denotes that $x_j$ is among the $k_n$-nearest neighbors of $x_k$.
Hence $H_{n,k_n}$ is a $k_n$ uniform hypergraph.

Again, if $\bar{\varepsilon}_n\gg \delta_n$, the $k_n$-NN graph is connected with probability one \cite{garcia2019variational}. Consequently, $|e_k^{(k_n)}|\geq 2$ for any $k=1,2,\cdots,n$ and hypergraph $H_{n,k_n}$ is also connected.

\subsection{Mathematical tools}
In this subsection, we present mathematical tools that will be needed for our study.
We follow the idea of \cite{slepcev2019analysis} and consider the convergence of minimizers of discrete functionals defined on hypergraphs to minimizers of continuum functionals.
It relies on properties of $\Gamma$-convergence \cite{dal2012introduction,braides2002gamma}, the optimal transportation from the probability measure $\mu$ to the associated empirical measure $\mu_n$ \cite{trillos2015rate}, and the $TL^p$ space \cite{garcia2016continuum}.

\begin{definition}
Let $X$ be a metric space and $F_n, F: X\rightarrow [-\infty,\infty]$ be functionals. We say that $F_n$ $\Gamma$-converges to $F$ as $n\rightarrow\infty$, denoted by $F_n\stackrel{\Gamma}{\longrightarrow }F$,
if for every $x\in X$,
\begin{itemize}
  \item Liminf inequality: For every sequence $\{x_n\}_{n\in\mathbb{N}}\subset X$ converging to $x$,
      \begin{equation*}
        \liminf_{n\rightarrow\infty} F_n(x_n)\geq F(x).
      \end{equation*}
  \item Limsup inequality: There exists a sequence $\{x_n\}_{n\in\mathbb{N}}\subset X$ converging to $x$ such that
      \begin{equation*}
        \limsup_{n\rightarrow\infty} F_n(x_n)\leq F(x).
      \end{equation*}
\end{itemize}

\end{definition}

The $\Gamma$-convergence implies the  convergence of minimizers, see \cite[Theorem 1.21]{braides2002gamma} for the proof.

\begin{proposition}\label{pr:2.2}
  Let $X$ be a metric space, $F_n, F: X\rightarrow [0,\infty]$ be functionals, and
  $F_n\stackrel{\Gamma}{\longrightarrow }F\not\equiv\infty$ as $n\rightarrow\infty$.
  If there exists a precompact sequence $\{x_n\}_{n\in\mathbb{N}}$ such that
  \begin{equation*}
    \lim_{n\rightarrow\infty}\left(F_n(x_n)-\inf_{x\in X}F_n(x)\right)=0,
  \end{equation*}
  then
  \begin{equation*}
    \lim_{n\rightarrow\infty}\inf_{x\in X}F_n(x)=\inf_{x\in X}F(x),
  \end{equation*}
  and any cluster point of $\{x_n\}_{n\in\mathbb{N}}$ is a minimizer of $F$.
\end{proposition}

Let $\mu_n=\frac{1}{n}\sum_{i=1}^n\delta_{x_i}$ be the empirical measure of $\Omega_n$. We call $T_n$ a transportation map from $\mu$ to $\mu_n$ if it satisfies the push-forward condition
\begin{equation*}
  \mu_n=T_{n\sharp}\mu=\mu\circ T_n^{-1}.
\end{equation*}
It implies the change of variables,
\begin{equation}\label{eq:cv1}
  \frac{1}{n}\sum_{i=1}^{n}\varphi(x_i)=\int_{\Omega}\varphi(x)d\mu_n(x)
  =\int_{\Omega}\varphi(T_n(x))d\mu(x)=\int_{\Omega}\varphi(T_n(x))\rho(x)dx,
\end{equation}
and
\begin{equation}\label{eq:cv2}
  \mu_n\text{-}\esssup_{x\in\Omega_n}\varphi(x)=\mu\text{-}\esssup_{y\in\Omega}\varphi(T_n(y))=\esssup_{y\in\Omega}\varphi(T_n(y)),
\end{equation}
for any function $\varphi:\Omega_n\rightarrow\mathbb{R}$.

The following proposition states the existence of the transportation map from $\mu$ to $\mu_n$ and its $L^\infty$ estimates \cite{slepcev2019analysis,trillos2015rate}.
\begin{proposition}\label{pr:2.3}
  Let $\{x_i\}_{i=1}^\infty$ be a sequence of independent random variables with distribution $\mu$ on $\Omega$. Then, almost surely there exist transportation maps $\{T_n\}_{n=1}^\infty$ from $\mu$ to $\mu_n$, such that
  \begin{align}\label{measure}
  \begin{split}
    c\leq\liminf_{n\rightarrow\infty}\frac{\|Id-T_n\|_{L^\infty(\Omega)}}{\delta_n}
    \leq\limsup_{n\rightarrow\infty}\frac{\|Id-T_n\|_{L^\infty(\Omega)}}{\delta_n}
    \leq C,
  \end{split}
  \end{align}
 where $\delta_n$ is defined in \eqref{delta}.
\end{proposition}

In the following, we do not use the almost sure statement but consider our problem in the deterministic setting. More precisely, we assume that the unlabeled points are the realization of the random variables $\{x_i\}_{i=1}^\infty$ mentioned in Proposition \ref{pr:2.3} such that conclusion \eqref{measure} holds.

To conclude this subsection, we recall the $TL^p$ space \cite{garcia2016continuum}, which provides an appropriate topology for the convergence of discrete functions to continuum functions.

\begin{definition}
  The space $TL^p(\Omega)$ is defined as
  \begin{equation*}
    TL^p(\Omega)=\{(\nu,g): \nu\in\mathcal{P}(\overline{\Omega}), g\in L^p(\nu)\},
  \end{equation*}
  where $\mathcal{P}(\overline{\Omega})$ is the space of probability measures.
  It is endowed with the $TL^p$-metric
  \begin{align*}
    d^p_{TL^p}((\nu_1,g_1),(\nu_2,g_2))=\inf_{\pi\in\Pi(\nu_1,\nu_2)}\left\{\iint_{\Omega\times\Omega}
    |x-y|^p+|g_1(x)-g_2(y)|^pd\pi(x,y)\right\},
  \end{align*}
  where $\Pi(\nu_1,\nu_2)$ is the set of transportation plans.
\end{definition}

The following proposition provides a characterization of the convergence of sequences in $TL^p(\Omega)$ \cite{garcia2016continuum}.
\begin{proposition}
  Let $(\nu,g)\in TL^p(\Omega)$ and $\{(\nu_n,g_n)\}_{n=1}^\infty\subset TL^p(\Omega)$, where $\nu$ is {absolutely continuous with respective to the Lebesgue measure}. Then $(\nu_n,g_n)\rightarrow (\nu,g)$ in $TL^p(\Omega)$ as $n\rightarrow\infty$ if and only if
  $\nu_n\rightarrow\nu$ weakly as $n\rightarrow\infty$ and
  there exists a sequence of transportation maps $\{T_n\}_{n=1}^\infty$ from $\nu$ to $\nu_n$ such that
  \begin{align*}
    \lim_{n\rightarrow\infty}\int_{\Omega}|x-T_n(x)|^pd\nu(x)=0,
  \end{align*}
  and
  \begin{align*}
    \lim_{n\rightarrow\infty}\int_{\Omega}|g(x)-g_n(T_n(x))|^pd\nu(x)=0.
  \end{align*}
\end{proposition}

In our case, the empirical measure $\mu_n$ converges weakly to $\mu$ as $n\rightarrow\infty$. Proposition \ref{pr:2.3} implies the existence of the transportation map $T_n$.
To conclude the convergence of $(\mu_n,g_n)$ to $(\mu,g)$ in $TL^p(\Omega)$ we only need to verify the convergence of $g_n\circ T_n$ to $g$ in $L^p(\Omega)$.

\section{Continuum Limit of the $p$-Laplacian Regularization on the $\varepsilon_n$-ball hypergraph}\label{se:3}

In this section, we study the continuum limit of the $p$-Laplacian regularization $\mathcal{E}^{con}_{n,\varepsilon_n}$ (see \eqref{eq:1.5} for its definition) on the $\varepsilon_n$-ball hypergraph $H_{n,\varepsilon_n}$.
The main result of this section is stated as follows.

\begin{theorem}\label{th:3.1}
  Let $p>d$ and $\delta_n\ll\varepsilon_n\ll1$. If $u_n$ is a minimizer of $\mathcal{E}^{con}_{n,\varepsilon_n}$, then almost surely,
  \begin{equation*}
    (\mu_n,u_n)\rightarrow (\mu,u),\quad\mbox{in } TL^p(\Omega),
  \end{equation*}
  as $n\rightarrow\infty$ for a function $u\in L^{p}(\Omega)$.
  Furthermore, $u$ is continuous,
  \begin{equation}\label{eq:th:3.1:1}
    \lim_{n\rightarrow\infty}\max_{x_k\in\Omega_n\cap\Omega'}|u_n(x_k)-u(x_k)|=0,
  \end{equation}
  for any $\Omega'$ compactly contained in $\Omega$, and
  $u$ is a minimizer of $\mathcal{E}^{con}$.
\end{theorem}

\begin{proof}
  We first claim that $\mathcal{E}^{con}_{n,\varepsilon_n}$ admits at least one minimizer $u_n$ for any fixed $n>0$ and $\mathcal{E}^{con}$ admits a unique minimizer.

  Let us define the vector space
  \begin{equation*}
    V=\{v:\Omega_n\rightarrow\mathbb{R}\},
  \end{equation*}
  and its subspace
  \begin{equation*}
    V_{\mathcal{O}}=\{v\in V: v(x_i)=y_i \mbox{ for } x_i\in\mathcal{O}\}.
  \end{equation*}
  It follows from the Poincar\'{e} type inequality (see \cite[Lemma 1]{el2020discrete} for $p=2$, the proof works for any $p>1$) that
  \begin{align*}
    \sum_{x_i\in\Omega_n\backslash\mathcal{O}}|u(x_i)|^p
    &\leq C\sum_{x_i\in\Omega_n}\sum_{x_j\in B(x_i,\varepsilon_n)}|u(x_i)-u(x_j)|^p+C\sum_{x_i\in\mathcal{O}}|u(x_i)|^p \\
    &\leq C\mathcal{E}_{n,\varepsilon_n}(u)+C\sum_{x_i\in\mathcal{O}}|u(x_i)|^p,
  \end{align*}
  and $\mathcal{E}_{n,\varepsilon_n}$ is coercive on $V_{\mathcal{O}}$.
  Clearly, it is also lower semi-continuous on $V_{\mathcal{O}}$.
  Then $\mathcal{E}_{n,\varepsilon_n}$ admits a minimizer $u_n$ on $V_{\mathcal{O}}$, which is also a minimizer of $\mathcal{E}^{con}_{n,\varepsilon_n}$.
  The existence of a unique minimizer for $\mathcal{E}^{con}$ can be verified similarly, see also \cite{slepcev2019analysis}.

Let $M= \max_{1\leq i\leq N}|y_i|$.
For any function $f:\Omega\rightarrow\mathbb{R}$, we have
\begin{equation*}
  \mathcal{E}^{con}_{n,\varepsilon_n}(f_M)\leq \mathcal{E}^{con}_{n,\varepsilon_n}(f),
\end{equation*}
where $f_M:=\max\{\min\{f,M\},-M\}$ denotes the truncation of $f$.
Consequently,
\begin{equation}\label{eq:th:3.1:2}
  \|u_n\|_{L^\infty(\Omega)}\leq M,
\end{equation}
for any $n>0$.
  Notice that
\begin{align}\label{eq:th:3.1:5}
  \begin{split}
\frac{1}{n^2\varepsilon_n^{p}}\sum_{i=1}^{n}\sum_{j=1}^{n}\chi_{B(x_i,\varepsilon_n)}&(x_j)
|u_n(x_i)-u_n(x_j)|^p \\
&\leq\frac{1}{n\varepsilon_n^{p}}\sum_{k=1}^{n}\max_{x_i,x_j\in B(x_k,\varepsilon_n)}|u_n(x_i)-u_n(x_j)|^p\\
&=\mathcal{E}_{n,\varepsilon_n}(u_n)\leq \mathcal{E}^{con}_{n,\varepsilon_n}(u_n)<\infty.
  \end{split}
\end{align}
The left-hand side is just the graph $p$-Laplacian regularization with the indicator function as the kernel.
By the compactness of the graph $p$-Laplacian regularization (see \cite[Proposition 4.4]{slepcev2019analysis}) and \eqref{eq:th:3.1:2}, there exists a subsequence of $\{u_n\}_{n\in\mathbb{N}}$, denoted by $\{u_{n_m}\}_{m\in\mathbb{N}}$, and a function $u\in L^p(\Omega)$, such that
  \begin{equation}\label{eq:th:3.1:3}
    (\mu_{n_m},u_{n_m})\rightarrow (\mu,u),\quad\mbox{in } TL^p(\Omega),
  \end{equation}
as $m\rightarrow \infty$.

Let us  assume that
  \begin{equation}\label{eq:th:3.1:4}
    \mathcal{E}_{n,\varepsilon_n}^{con}\stackrel{\Gamma}{\longrightarrow }
      \mathcal{E}^{con},
  \end{equation}
  in the $TL^p$ metric on the set $\{(\nu,g): \nu\in\mathcal{P}(\Omega), \|g\|_{L^\infty(\nu)}\leq M\}$,
 which will be proven in the following in Theorem \ref{th:3.8}. Then it follows from Proposition \ref{pr:2.2} that $u$ is a minimizer of $\mathcal{E}^{con}$.
Since the minimizer of $\mathcal{E}^{con}$ is unique, the convergence in \eqref{eq:th:3.1:3} holds along the whole sequence.
The uniform convergence \eqref{eq:th:3.1:1} will be proven in Lemma \ref{le:3.7}.
\end{proof}

\begin{remark}
  If $1<p\leq d$, we still have \eqref{eq:th:3.1:3}. But now $\mathcal{E}^{con}=\mathcal{E}$ and \eqref{eq:th:3.1:4} becomes
  \begin{equation*}
    \mathcal{E}_{n,\varepsilon_n}^{con}\stackrel{\Gamma}{\longrightarrow }
      \mathcal{E}.
  \end{equation*}
  By Proposition \ref{pr:2.2}, $u$ is a minimizer of $\mathcal{E}$. In this case we no longer have the uniform convergence \eqref{eq:th:3.1:1}.
\end{remark}

The rest of this section is devoted to the proof of the $\Gamma$-convergence \eqref{eq:th:3.1:4} and the uniform convergence \eqref{eq:th:3.1:1}.

\subsection{Nonlocal to local convergence}
In this section,
we introduce the nonlocal functional
\begin{equation*}
  \mathcal{E}_{\varepsilon}(u)=
  \int_{\Omega}\frac{1}{(\varepsilon(x))^{p}}\esssup_{y,z\in B(x,\varepsilon(x))\cap\Omega}|u(z)-u(y)|^p\rho(x)dx
\end{equation*}
as a bridge between $\mathcal{E}_{n,\varepsilon_n}$ and $\mathcal{E}$, and prove the $\Gamma$-convergence of $\mathcal{E}_{\varepsilon}$ to $\mathcal{E}$ in $L^p(\Omega)$ as $\varepsilon\rightarrow 0$.
Here we consider a general setting where the connection radius $\varepsilon$ depends on the location $x\in\Omega$.
This is helpful for Section \ref{se:4} in which we study the continuum limit of the $p$-Laplacian regularization on the $k_n$-NN hypergraph.

The $\Gamma$-convergence of $\mathcal{E}_{\varepsilon}$ to $\mathcal{E}$ in $L^p(\Omega)$ as $\varepsilon\rightarrow 0$ is understood in the following sense. For any sequence of functions $\{\varepsilon_n(x)\}_{n\in\mathbb{N}}$
such that
\begin{equation}\label{eq:vare1}
  1\leq\frac{\sup_{x\in\Omega}\varepsilon_n(x)}{\inf_{x\in\Omega}\varepsilon_n(x)}\leq C<\infty,\quad \forall n>0,
\end{equation}
{and}
\begin{equation*}
  0<\inf_{x\in\Omega}\varepsilon_n(x)
  \leq \sup_{x\in\Omega}\varepsilon_n(x) \rightarrow 0,
\end{equation*}
as $n\rightarrow\infty$, we have
\begin{equation*}
  \mathcal{E}_{\varepsilon_n}\stackrel{\Gamma}{\longrightarrow }
    \mathcal{E},
\end{equation*}
in $L^p(\Omega)$ as $n\rightarrow\infty$.
In this subsection, we simply use $\varepsilon\rightarrow 0$ to represent that $\sup_{x\in\Omega}\varepsilon_n(x)\rightarrow 0$ as $n\rightarrow\infty$ for any sequence of functions $\{\varepsilon_n(x)\}_{n\in\mathbb{N}}$.

  The characteristic function $\eta: \mathbb{R}^d\rightarrow\mathbb{R}$ defined by
  \begin{equation*}
    \eta(x)=
    \begin{cases}
      1,\quad \mbox{if } |x|\leq 1,\\
      0,\quad \mbox{if }  |x|>1,
    \end{cases}
  \end{equation*}
  will be used several times in the following.

\begin{lemma}\label{le:3.2}
  Let $\{u_\varepsilon\}$ be a sequence of uniformly bounded $C^2$ functions.
  If $\nabla u_\varepsilon\rightarrow \nabla u$ in $L^p(\Omega)$ as $\varepsilon\rightarrow 0$, then
  \begin{equation*}
    \lim_{\varepsilon\rightarrow 0} \mathcal{E}_\varepsilon(u_\varepsilon)=\mathcal{E}(u).
  \end{equation*}
\end{lemma}

\begin{proof}
  By change of variables $y=x+\varepsilon(x) \hat{y}$, $z=x+\varepsilon(x) \hat{z}$ and Taylor's expansion,
  \begin{align*}
    \mathcal{E}_\varepsilon(u_\varepsilon)
    &=
    \int_{\Omega}\frac{1}{(\varepsilon(x))^{p}}\esssup_{y,z\in\mathbb{R}^d}
    \eta\left(\frac{x-y}{\varepsilon(x)}\right)\eta\left(\frac{x-z}{\varepsilon(x)}\right)
    |u_\varepsilon(z)-u_\varepsilon(y)|^p \rho(x)dx -I_1\\
    &=
    \int_{\Omega}\frac{1}{(\varepsilon(x))^{p}}\esssup_{\hat{y},\hat{z}\in\mathbb{R}^d}
    \eta(\hat{y})\eta(\hat{z})
    |u_\varepsilon(x+\varepsilon(x) \hat{z})-u_\varepsilon(x+\varepsilon(x) \hat{y})|^p \rho(x)dx -I_1\\
    &=\int_{\Omega}\esssup_{\hat{y},\hat{z}\in\mathbb{R}^d}
    \eta(\hat{y})\eta(\hat{z})
    |\nabla u_\varepsilon(x)\cdot (\hat{z}-\hat{y})+\mathcal{O}(\varepsilon(x))|^p \rho(x)dx -I_1\\
    &=I_2-I_1,
  \end{align*}
where
\begin{align*}
  |I_1|=&\left|
    \int_{\Omega}\frac{1}{(\varepsilon(x))^{p}}\esssup_{y,z\in\mathbb{R}^d}
    \eta\left(\frac{x-y}{\varepsilon(x)}\right)\eta\left(\frac{x-z}{\varepsilon(x)}\right)
    |u_\varepsilon(z)-u_\varepsilon(y)|^p \rho(x)dx\right. \\
    &\left.-
    \int_{\Omega}\frac{1}{(\varepsilon(x))^{p}}\esssup_{y,z\in\Omega}
    \eta\left(\frac{x-y}{\varepsilon(x)}\right)\eta\left(\frac{x-z}{\varepsilon(x)}\right)
    |u_\varepsilon(z)-u_\varepsilon(y)|^p \rho(x)dx\right|\\
    \leq& C
    \int_{\Omega}\esssup_{y\in\mathbb{R}^d, z\in\mathbb{R}^d\backslash\Omega}
    \eta\left(\frac{x-y}{\varepsilon(x)}\right)\eta\left(\frac{x-z}{\varepsilon(x)}\right) dx,
\end{align*}
converges to zero as $\varepsilon\rightarrow 0$.
  Consequently,
  \begin{align*}
    \lim_{\varepsilon\rightarrow 0}\mathcal{E}_\varepsilon(u_\varepsilon)
    =\lim_{\varepsilon\rightarrow 0} I_2
    &=
    \int_{\Omega}\esssup_{\hat{y},\hat{z}\in\mathbb{R}^d}
    \eta(\hat{y})\eta(\hat{z})
    |\nabla u(x)\cdot (\hat{z}-\hat{y})|^p \rho(x)dx \\
    &=2^p\int_{\Omega}|\nabla u|^p  \rho dx=\mathcal{E}(u).
  \end{align*}
\end{proof}

Lemma \ref{le:3.2} implies the liminf inequality of the $\Gamma$-convergence, which is stated as follows.
\begin{lemma}\label{le:3.3}
  If $u_\varepsilon\rightarrow u$ in $L^p(\Omega)$ as $\varepsilon\rightarrow 0$, then
  \begin{equation*}
    \liminf_{\varepsilon\rightarrow 0} \mathcal{E}_\varepsilon(u_\varepsilon)\geq \mathcal{E}(u).
  \end{equation*}
\end{lemma}
\begin{proof}
  Define $\Omega_\delta=\{x\in\Omega: \text{dist}(x,\partial\Omega)>\delta\}$ for any $\delta>0$, $\varepsilon_1=\inf_{x\in\Omega}\varepsilon(x)$, and $\varepsilon_2=\sup_{x\in\Omega}\varepsilon(x)$.
  {Assume w.l.o.g. that $\mathcal{E}_\varepsilon(u_\varepsilon)<\infty$}. For any $\alpha,\beta\in\mathbb{R}^d$ with $|\alpha|=|\beta|=r\in (0,1)$, it follows from \eqref{eq:vare1} that
  \begin{equation*}
    \int_{\Omega_{\varepsilon_1}}\left|\frac{u_\varepsilon(x+\varepsilon_1 \alpha)-u_\varepsilon(x+\varepsilon_1 \beta)}{\varepsilon_1}\right|^p dx
    \leq \frac{1}{\inf \rho}\left(\frac{\varepsilon_2}{\varepsilon_1}\right)^p\mathcal{E}_\varepsilon(u_\varepsilon)<\infty.
  \end{equation*}
  Then there exists a function $g\in L^p(\Omega)$ such that (up to a subsequence),
  \begin{equation*}
    \chi_{\Omega_{\varepsilon_1}}(\cdot)\frac{u_\varepsilon(\cdot+\varepsilon_1 \alpha)-u_\varepsilon(\cdot+\varepsilon_1 \beta)}{\varepsilon_1}\rightharpoonup g,\quad \mbox{in } L^p(\Omega),
  \end{equation*}
as $\varepsilon\rightarrow 0$. Namely, for any $\varphi(x)\in C_0^\infty(\Omega)$,
if $\varepsilon_1$ is sufficiently small,
\begin{align*}
  \int_{\Omega_{\varepsilon_1}}\frac{u_\varepsilon(x+\varepsilon_1 \alpha)-u_\varepsilon(x+\varepsilon_1 \beta)}{\varepsilon_1} \varphi(x)dx
  =\int_{\Omega}u_\varepsilon(x)\frac{\varphi(x-\varepsilon_1 \alpha)-\varphi(x-\varepsilon_1 \beta)}{\varepsilon_1}dx \\
  \rightarrow -\int_{\Omega}u(x)\nabla\varphi(x)\cdot(\alpha-\beta)dx
  =\int_{\Omega}g(x)\varphi(x)dx.
\end{align*}
By choosing $\alpha=re_i$ and $\beta=-re_i$, we have
\begin{equation*}
  -2r\int_{\Omega}u(x)\frac{\partial}{\partial x_i}\varphi(x) dx=\int_{\Omega}g(x)\varphi(x)dx,
\end{equation*}
which implies that $u\in W^{1,p}(\Omega)$.

Assume that $J: \mathbb{R}^d\rightarrow[0,\infty)$ is a standard mollifier with $\text{supp} J\subset \overline{B(0,1)}$ and $\int_{\mathbb{R}^d}J(s)ds=1$. Define $J_\delta(s)=\frac{1}{\delta^d}J(s/\delta)$. Then for a small $\delta>0$ and a sufficiently small $\varepsilon>0$, it follows from a change of variables $\hat{y}=y+s$, $\hat{z}=z+s$, and Jensen's inequality that
  \begin{align*}
    \mathcal{E}_{\varepsilon}(u_\varepsilon)
    &=
  \int_{\Omega}\frac{1}{(\varepsilon(x))^{p}}\esssup_{y,z\in B(x,\varepsilon(x))\cap\Omega}|u_\varepsilon(z)-u_\varepsilon(y)|^p \rho(x)dx\\
  &=
  \int_{\mathbb{R}^d}J_\delta(s)\int_{\Omega}\frac{1}{(\varepsilon(x))^{p}}\esssup_{y,z\in B(x,\varepsilon(x))\cap\Omega}|u_\varepsilon(z)-u_\varepsilon(y)|^p\rho(x)dxds\\
  &\geq
  \int_{\mathbb{R}^d}\int_{\Omega_{2\delta}}J_\delta(s)\frac{1}{(\varepsilon(x))^{p}}\esssup_{\hat{y},\hat{z}\in B(x+s,\varepsilon(x))}|u_\varepsilon(\hat{z}-s)-u_\varepsilon(\hat{y}-s)|^p\rho(x)dxds\\
  &\geq
  \int_{\Omega_{2\delta}}\frac{1}{(\varepsilon(x))^{p}}
  \inf_{|r|\leq\delta}\esssup_{\hat{y},\hat{z}\in B(x+r,\varepsilon(x))}
  \int_{\mathbb{R}^d}J_\delta(s)|u_\varepsilon(\hat{z}-s)-u_\varepsilon(\hat{y}-s)|^pds\rho(x)dx\\
  &\geq
  \int_{\Omega_{2\delta}}\frac{1}{(\varepsilon(x))^{p}}
  \inf_{|r|\leq\delta}\esssup_{\hat{y},\hat{z}\in B(x+r,\varepsilon(x))}
\left|\int_{\mathbb{R}^d}J_\delta(s)\left(u_\varepsilon(\hat{z}-s)-u_\varepsilon(\hat{y}-s)\right)ds\right|^p\rho({x})d{x}\\
&=\int_{\Omega_{2\delta}}\frac{1}{(\varepsilon(x))^{p}}
\inf_{|s|\leq\delta}\esssup_{\hat{y},\hat{z}\in B(x+s,\varepsilon(x))}
\left|u_{\varepsilon,\delta}(\hat{z})-u_{\varepsilon,\delta}(\hat{y})\right|^p\rho({x})d{x}
  \end{align*}
where $u_{\varepsilon,\delta}=(u_\varepsilon)_\delta=J_\delta*u_\varepsilon$.
Clearly,
\begin{equation*}
  \nabla u_{\varepsilon,\delta}\rightarrow\nabla u_\delta, \quad \text{in } L^p(\Omega),
\end{equation*}
 as $\varepsilon\rightarrow 0$ for a fixed $\delta>0$,
 \begin{equation*}
   \|\nabla u_{\varepsilon,\delta}\|_{L^\infty(\Omega)}
   +\|D^2 u_{\varepsilon,\delta}\|_{L^\infty(\Omega)}\leq C,
 \end{equation*}
 for a constant $C$ depending not on $\varepsilon$,
  and
 \begin{equation*}
  \nabla u_{\delta}\rightarrow\nabla u, \quad \text{in } L^p(\Omega),
\end{equation*}
 as $\delta\rightarrow 0$. It is possible to consider a subsequence of $\{u_\delta\}$ such that the above convergence holds almost everywhere.

It follows from the proof of Lemma \ref{le:3.2} and Lebesgue's dominated convergence theorem that
\begin{align*}
  &\liminf_{\varepsilon\rightarrow 0} \mathcal{E}_\varepsilon(u_\varepsilon)\\
  &\geq \lim_{\delta\rightarrow 0}\liminf_{\varepsilon\rightarrow 0}
  \int_{\Omega_{2\delta}}\frac{1}{(\varepsilon(x))^{p}}
\inf_{|s|\leq\delta}\esssup_{\hat{y},\hat{z}\in B(x+s,\varepsilon(x))}
\left|u_{\varepsilon,\delta}(\hat{z})-u_{\varepsilon,\delta}(\hat{y})\right|^p\rho({x})d{x}\\
  &=\lim_{\delta\rightarrow 0}2^p\int_{\Omega_{2\delta}}\inf_{|s|\leq\delta}|\nabla u_\delta(x+s)|^p\rho(x) dx
  =2^p \int_{\Omega}|\nabla u|^p\rho dx
  =\mathcal{E}(u).
\end{align*}
\end{proof}

We now prove the limsup inequality of the $\Gamma$-convergence.
\begin{lemma}\label{le:3.4}
  For any $u\in L^p(\Omega)$, there exists a sequence of functions $\{u_\varepsilon\}\subset L^p(\Omega)$ such that $u_\varepsilon\rightarrow u$ in $L^p(\Omega)$ as $\varepsilon\rightarrow 0$ and
  \begin{equation*}
    \limsup_{\varepsilon\rightarrow 0} \mathcal{E}_\varepsilon(u_\varepsilon)\leq  \mathcal{E}(u).
  \end{equation*}
\end{lemma}
\begin{proof}
  Assume w.l.o.g. that $\mathcal{E}(u)<\infty$, i.e., $u\in W^{1,p}(\Omega)$.
  By \cite[Remark 2.7]{garcia2016continuum},
  we only need to prove the limsup inequality for all $u$ in a dense subset of $W^{1,p}(\Omega)$, e.g., $C^2(\overline{\Omega})\cap W^{1,p}(\Omega)$.

  Let $u\in C^2(\overline{\Omega})\cap W^{1,p}(\Omega)$ and $u_\varepsilon=u$.
  Define $\Omega^\varepsilon=\{x\in\mathbb{R}^d: \text{dist}(x,\Omega)<\sup_{x\in\Omega}\varepsilon(x)\}$.
  By extension, $u$ is well-defined on $\Omega^\varepsilon$.
  By change of variables $y=x+\varepsilon(x) \hat{y}$, $z=x+\varepsilon(x) \hat{z}$, and Taylor's expansion,
  \begin{align*}
    \mathcal{E}_\varepsilon(u)
    &=
    \int_{\Omega}\frac{1}{(\varepsilon(x))^{p}}\esssup_{y,z\in B(x,\varepsilon(x))\cap\Omega}|u(z)-u(y)|^p\rho(x)dx \\
    &=\int_{\Omega}\frac{1}{(\varepsilon(x))^{p}}\esssup_{y,z\in B(x,\varepsilon(x))\cap\Omega}\left|\int_0^1\nabla u(y+t(z-y))\cdot (z-y)dt\right|^p\rho(x) dx \\
    &\leq\int_{\Omega}\esssup_{\hat{y},\hat{z}\in B(0,1)}\left|\int_0^1\nabla u(x+\varepsilon(x) \hat{y}+\varepsilon(x) t(\hat{z}-\hat{y}))\cdot (\hat{z}-\hat{y})dt\right|^p \rho(x)dx \\
    &\leq \int_{\Omega}\esssup_{\hat{y},\hat{z}\in B(0,1)}\left|\nabla u(x)\cdot (\hat{z}-\hat{y})\right|^p\rho(x)dx + C\sup_{x\in\Omega}(\varepsilon(x))^p,
  \end{align*}
  where the constant $C$ depends on $\|D^2 u\|_{L^\infty(\Omega)}$.
  We conclude that
  \begin{equation*}
    \limsup_{\varepsilon\rightarrow 0} \mathcal{E}_\varepsilon(u)\leq  \mathcal{E}(u),
  \end{equation*}
  and completes the proof.
\end{proof}

Combining Lemma \ref{le:3.3} and Lemma \ref{le:3.4} we arrive at the $\Gamma$-convergence of $\mathcal{E}_{\varepsilon}$ to $\mathcal{E}$.
\begin{theorem}\label{th:3.5}
  Let $1<p<\infty$. Then
  \begin{equation*}
    \mathcal{E}_{\varepsilon}\stackrel{\Gamma}{\longrightarrow }\mathcal{E},
  \end{equation*}
  in $L^p(\Omega)$ as $\varepsilon\rightarrow 0$.
\end{theorem}

\subsection{Discrete to continuum convergence}
The $\Gamma$-convergence of nonlocal functional $\mathcal{E}_{\varepsilon}$ implies the $\Gamma$-convergence of discrete functional $\mathcal{E}_{n,\varepsilon_n}$.
This is achieved by rewriting the discrete functional as a nonlocal functional via the transportation map $T_n$ introduced in Proposition \ref{pr:2.3}.

\begin{theorem}\label{th:3.6}
  Let $1<p<\infty$ and $\delta_n\ll\varepsilon_n\ll1$. Then with probability one,
  \begin{equation*}
    \mathcal{E}_{n,\varepsilon_n}\stackrel{\Gamma}{\longrightarrow }\mathcal{E},
  \end{equation*}
  in $TL^p(\Omega)$ as $n\rightarrow\infty$.
\end{theorem}
\begin{proof}
  Let
  \begin{equation*}
    \tilde{\varepsilon}_n:=\varepsilon_n+2\|T_n-Id\|_{L^\infty(\Omega)},\quad
    \hat{\varepsilon}_n:=\varepsilon_n-2\|T_n-Id\|_{L^\infty(\Omega)}.
  \end{equation*}
For any $x,y\in\Omega$,
\begin{align*}
  |T_n(x)-T_n(y)|\leq \varepsilon_n \Longrightarrow  |x-y|\leq \tilde{\varepsilon}_n,\quad
  |x-y|\leq \hat{\varepsilon}_n
  \Longrightarrow |T_n(x)-T_n(y)|\leq \varepsilon_n.
\end{align*}
It follows that
\begin{equation}\label{eq:se:3.1:1}
    \eta\left(\frac{x-y}{\hat{\varepsilon}_n}\right)
    \leq\eta\left(\frac{T_n(x)-T_n(y)}{{\varepsilon}_n}\right)
    \leq \eta\left(\frac{x-y}{\tilde{\varepsilon}_n}\right).
\end{equation}
Moreover, \eqref{measure} yields
\begin{equation*}
  \lim_{n\rightarrow\infty}\frac{\hat{\varepsilon}_n}{\varepsilon_n}
  =\lim_{n\rightarrow\infty}\frac{\tilde{\varepsilon}_n}{\varepsilon_n}
  =1.
\end{equation*}

  Notice from \eqref{eq:cv1} and \eqref{eq:cv2} that
  \begin{align*}
    &\mathcal{E}_{n,\varepsilon_n}(u_n)\\
    &=\frac{1}{\varepsilon_n^{p}}\int_{\Omega}\esssup_{y,z\in\Omega}
    \eta\left(\frac{T_n(x)-T_n(z)}{\varepsilon_n}\right)
    \eta\left(\frac{T_n(x)-T_n(y)}{\varepsilon_n}\right)
    |\tilde{u}_n(z)-\tilde{u}_n(y)|^p \rho(x)dx,
  \end{align*}
  where $\tilde{u}_n=u_n\circ T_n$.

Let us consider the liminf inequality.
Assume that $(\mu_n, u_n)\rightarrow (\mu,u)$ in $TL^p(\Omega)$, i.e., $\tilde{u}_n\rightarrow u$ in $L^p(\Omega)$.
It follows from \eqref{eq:se:3.1:1} that
\begin{align*}
  \mathcal{E}_{n,\varepsilon_n}(u_n)\geq \frac{1}{\varepsilon_n^{p}}\int_{\Omega}\esssup_{y,z\in\Omega}
  \eta\left(\frac{x-z}{\hat{\varepsilon}_n}\right)
  \eta\left(\frac{x-y}{\hat{\varepsilon}_n}\right)
  |\tilde{u}_n(z)-\tilde{u}_n(y)|^p\rho(x)dx.
\end{align*}
Then by Lemma \ref{le:3.3},
\begin{align*}
  \liminf_{n\rightarrow\infty}\mathcal{E}_{n,\varepsilon_n}(u_n)
  \geq  \liminf_{n\rightarrow\infty} \left(\frac{\hat{\varepsilon}_n}{\varepsilon_n}\right)^{p} \mathcal{E}_{\hat{\varepsilon}_n}(\tilde{u}_n)\geq \mathcal{E}(u).
\end{align*}

For the limsup inequality, we
assume that $u\in W^{1,p}(\Omega)$ is Lipschitz. Let $u_n$ be the restriction of $u$ to the first $n$ data points. It follows from \eqref{eq:se:3.1:1} that
\begin{align*}
  \mathcal{E}_{n,\varepsilon_n}(u_n)
  &\leq \frac{1}{\varepsilon_n^{p}}\int_{\Omega}\esssup_{y,z\in\Omega}
  \eta\left(\frac{x-z}{\tilde{\varepsilon}_n}\right)
  \eta\left(\frac{x-y}{\tilde{\varepsilon}_n}\right)
  |\tilde{u}_n(z)-\tilde{u}_n(y)|^p\rho(x)dx \\
  &=\frac{1}{\varepsilon_n^{p}}\int_{\Omega}\esssup_{y,z\in\Omega}
  \eta\left(\frac{x-z}{\tilde{\varepsilon}_n}\right)
  \eta\left(\frac{x-y}{\tilde{\varepsilon}_n}\right)
  |u(z)-u(y)|^p\rho(x)dx + I,
\end{align*}
where, according to \eqref{measure},
\begin{align*}
  |I|&=\frac{1}{\varepsilon_n^{p}}\left|\int_{\Omega}\esssup_{y,z\in\Omega}
  \eta\left(\frac{x-z}{\tilde{\varepsilon}_n}\right)
  \eta\left(\frac{x-y}{\tilde{\varepsilon}_n}\right)
  \left(|\tilde{u}_n(z)-\tilde{u}_n(y)|^p-|u(z)-u(y)|^p\right)\rho(x)dx\right|\\
  &\leq \frac{C}{\varepsilon_n^{p}}\int_{\Omega}\esssup_{z\in\Omega}
  \eta\left(\frac{x-z}{\tilde{\varepsilon}_n}\right)
   |\tilde{u}_n(z)-u(z)|^pdx\\
  &\leq \frac{C}{\varepsilon_n^{p}}\int_{\Omega}\esssup_{z\in\Omega}
  \eta\left(\frac{x-z}{\tilde{\varepsilon}_n}\right)
   |T_n(z)-z|^pdx\\
  &\leq \frac{C}{\varepsilon_n^{p}}\|T_n-Id\|_{L^\infty(\Omega)}^p
\end{align*}
goes to zero as $n\rightarrow \infty$.
The proof is completed by Lemma \ref{le:3.4}. Namely,
\begin{align*}
  \limsup_{n\rightarrow\infty}\mathcal{E}_{n,\varepsilon_n}(u_n)
  \leq\limsup_{n\rightarrow\infty}\left(\frac{\tilde{\varepsilon}_n}{\varepsilon_n}\right)^{p}\mathcal{E}_{\tilde{\varepsilon}_n}(u)
  \leq \mathcal{E}(u).
\end{align*}
\end{proof}

We now go back to consider the constrained functionals. The following Lemma plays a key role.
\begin{lemma}\label{le:3.7}
  Let $\{u_n\}_{n\in\mathbb{N}}$ be a sequence of functions on $\Omega_n$ such that
  \begin{equation*}
    \sup_{n\in\mathbb{N}}\mathcal{E}_{n,\varepsilon_n}(u_n)<\infty.
  \end{equation*}
  If $p>d$ and $(\mu_n,u_n)\rightarrow (\mu,u)$ in $TL^p(\Omega)$, then
  \begin{equation*}
    \lim_{n\rightarrow\infty}\max_{x_k\in\Omega_n\cap\Omega'}|u_n(x_k)-u(x_k)|=0,
  \end{equation*}
for all $\Omega'$ compactly supported in $\Omega$.
\end{lemma}
\begin{proof}
  For any fixed $x_k\in\Omega_n$,
  we have
  \begin{align*}
    \frac{1}{n\varepsilon_n^{p}}\sum_{j=1}^{n}
    \eta\left(\frac{x_k-x_j}{\varepsilon_n}\right)
    |u_n(x_k)-u_n(x_j)|^p
    &\leq\frac{1}{n\varepsilon_n^{p}}\sum_{j=1}^{n}\max_{x_k\in B(x_j,\varepsilon_n)}|u_n(x_k)-u_n(x_j)|^p\\
    &\leq \mathcal{E}_{n,\varepsilon_n}(u_n)<\infty,
  \end{align*}
 which can be rewritten as
  \begin{equation*}
    \frac{1}{\varepsilon_n^p}\int_\Omega
    \eta\left(\frac{x_k-T_n(y)}{\varepsilon_n}\right)
    |\tilde{u}_n(y)-u_n(x_k)|^p\rho(y)dy<\infty.
  \end{equation*}
  Let $J$ be a standard mollifier and $n$ be large such that
  \begin{equation*}
    J_{\hat{\varepsilon}_n}(x_k-y)=\frac{1}{\hat{\varepsilon}_n^d}J\left(\frac{x_k-y}{\hat{\varepsilon}_n}\right)
    \leq \frac{C}{\hat{\varepsilon}_n^d}
    \eta\left(\frac{x_k-y}{\hat{\varepsilon}_n}\right)
    \leq \frac{C}{\hat{\varepsilon}_n^d}
    \eta\left(\frac{x_k-T_n(y)}{{\varepsilon}_n}\right),
    \quad \forall y\in\Omega.
  \end{equation*}
We have
  \begin{equation*}
    \frac{\hat{\varepsilon}_n^d}{\varepsilon_n^p}\int_\Omega J_{\hat{\varepsilon}_n}(x_k-y)|\tilde{u}_n(y)-u_n(x_k)|^pdy<\infty.
  \end{equation*}
  Consequently, by $p>d$ and Jensen's inequality,
  \begin{equation}\label{eq:le:3.7:1}
    \int_\Omega J_{\hat{\varepsilon}_n}(x_k-y)|\tilde{u}_n(y)-u_n(x_k)|dy
    \leq \left(\int_\Omega J_{\hat{\varepsilon}_n}(x_k-y)|\tilde{u}_n(y)-u_n(x_k)|^pdy\right)^{1/p}
    \rightarrow 0,
  \end{equation}
  as $n\rightarrow\infty$.

  The uniform boundedness of the graph $p$-Laplacian regularization (i.e., \eqref{eq:th:3.1:5}) and the proof of Lemma 4.5 of \cite{slepcev2019analysis} imply that $u\in W^{1,p}(\Omega')\hookrightarrow C^{1-\frac{n}{p}}(\overline{\Omega'})$ and
  \begin{equation}\label{eq:le:3.7:2}
    J_{\hat{\varepsilon}_n}*\tilde{u}_n\rightarrow u,\quad \mbox{uniformly on } \Omega',
  \end{equation}
  as $n\rightarrow\infty$.
  Then the proof is completed by \eqref{eq:le:3.7:1}--\eqref{eq:le:3.7:2}, i.e.,
  for any $x_k\in\Omega_n\cap\Omega'$,
  \begin{align*}
    |u_n(x_k)&-u(x_k)|
    =\left|u_n(x_k)-\int_{\Omega}J_{\hat{\varepsilon}_n}(x_k-y)\tilde{u}_n(y)dy\right.\\
    &\qquad\qquad\qquad\qquad\qquad\qquad\qquad\qquad+\left.\int_{\Omega}J_{\hat{\varepsilon}_n}(x_k-y)\tilde{u}_n(y)dy-u(x_k)\right|\\
    &\leq \int_{\Omega}J_{\hat{\varepsilon}_n}(x_k-y)|\tilde{u}_n(y)-u_n(x_k)|dy
    +\left|\int_{\Omega}J_{\hat{\varepsilon}_n}(x_k-y)\tilde{u}_n(y)dy-u(x_k)\right|
  \end{align*}
  converges to zero as $n\rightarrow\infty$.
\end{proof}

The $\Gamma$-convergence of the constrained functional $\mathcal{E}_{n,\varepsilon_n}^{con}$ follows from Theorem \ref{th:3.6} and Lemma \ref{le:3.7} directly.

\begin{theorem}\label{th:3.8}
  Let $p>d$ and $\delta_n\ll\varepsilon_n\ll1$. Then with probability one,
  \begin{equation*}
    \mathcal{E}_{n,\varepsilon_n}^{con}\stackrel{\Gamma}{\longrightarrow }
      \mathcal{E}^{con},
  \end{equation*}
  in the $TL^p$ metric on the set $\{(\nu,g): \nu\in\mathcal{P}(\Omega), \|g\|_{L^\infty(\nu)}\leq M\}$, where $M= \max_{1\leq i\leq N}|y_i|$.
\end{theorem}

\begin{proof}
  If $(\mu_n,u_n)\rightarrow (\mu,u)$ in $TL^p(\Omega)$ with $\|u_n\|_{L^\infty(\mu_n)}\leq M$  and $\|u\|_{L^\infty(\mu)}\leq M$, then
  \begin{equation*}
    \liminf_{n\rightarrow\infty}\mathcal{E}_{n,\varepsilon_n}^{con}(u_n)
    \geq \liminf_{n\rightarrow\infty}\mathcal{E}_{n,\varepsilon_n}(u_n)
    \geq \mathcal{E}(u)=\mathcal{E}^{con}(u),
  \end{equation*}
  where the last equality follows from Lemma \ref{le:3.7}.
 This proves the liminf inequality.

  Now let us consider the limsup inequality.
  Let $u\in W^{1,p}(\Omega)$ with $\|u\|_{L^\infty(\mu)}\leq M$.
  We further assume that $u(x_i)=y_i$ for $x_i\in\mathcal{O}$. Otherwise, there is nothing to prove.
  Define $u_n(x_i)=u(x_i)$ for $x_i\in\Omega_n$. Then
\begin{align*}
  \limsup_{n\rightarrow\infty}\mathcal{E}_{n,\varepsilon_n}^{con}(u_n)
  =\limsup_{n\rightarrow\infty}\mathcal{E}_{n,\varepsilon_n}(u_n)
  \leq \mathcal{E}(u)=\mathcal{E}^{con}(u).
\end{align*}
This completes the proof.
\end{proof}

\section{Continuum Limit of the $p$-Laplacian Regularization on the $k_n$-NN hypergraph}\label{se:4}
In this section, we generalize the previous results to the hypergraph $p$-Laplacian regularization on the $k_n$-NN hypergraph.
Let us begin with the $\Gamma$-convergence of the unconstrained functional $\mathcal{F}_{n,k_n}$, which is defined in \eqref{eq:1.7}.


\begin{theorem}\label{th:4.1}
  Let $1<p<\infty$ and $\delta_n\ll\bar{\varepsilon}_n\ll1$, where $\bar{\varepsilon}_n$ is defined in \eqref{eq:1.8}. Then with probability one,
  \begin{equation*}
    \mathcal{F}_{n,k_n}\stackrel{\Gamma}{\longrightarrow }\mathcal{E}(~\cdot~ ;\rho^{1-p/d}),
  \end{equation*}
  in $TL^p(\Omega)$ as $n\rightarrow\infty$.
\end{theorem}
\begin{proof}
1. Let $\varepsilon_n(x)$ be a function of $x\in\Omega$ such that
\begin{equation}\label{eq:th:4.1:1}
  \mu\left(B(x,\varepsilon_n(x))\right)=\frac{k_n}{n}.
\end{equation}
We first claim that \eqref{eq:vare1} holds.
Since $\Omega$ is bounded and has Lipschitz boundary, it satisfies the cone condition \cite{adams2003sobolev}.
More precisely, there exists a finite cone $V$ such that each $x\in\Omega$ is the vertex of a finite cone $V_x$ contained in $\Omega$ and congruent to $V$.
If $n$ is large, it follows from \eqref{eq:th:4.1:1} and \eqref{eq:1.8} that
\begin{equation*}
  \alpha_d(\bar{\varepsilon}_n)^d=\mu\left(B(x,\varepsilon_n(x))\right)\geq |V_x\cap B(x,\varepsilon_n(x))|\geq C|B(x,\varepsilon_n(x))|
  =C\alpha_d({\varepsilon}_n(x))^d,
\end{equation*}
for any $x\in\Omega$, where $C$ depends only on $V$.
Consequently,
\begin{equation*}
  \bar{\varepsilon}_n\leq \varepsilon_n(x)\leq C^{-\frac{1}{d}}\bar{\varepsilon}_n, \quad x\in\Omega,
\end{equation*}
which proves \eqref{eq:vare1} and
\begin{equation*}
  1\gg\varepsilon_n(x)\geq \bar{\varepsilon}_n\gg\delta_n.
\end{equation*}

Define
\begin{equation*}
  \hat{\varepsilon}_n(x):=\varepsilon_n(x) - 4\|T_n-Id\|_{L^\infty(\Omega)},\quad
  \tilde{\varepsilon}_n(x):=\varepsilon_n(x) + 4\|T_n-Id\|_{L^\infty(\Omega)}.
\end{equation*}
The previous result and \eqref{measure} yield
\begin{equation}\label{eq:th:4.1:2}
  \lim_{n\rightarrow\infty}\frac{\hat{\varepsilon}_n(x)}{\varepsilon_n(x)}
  =\lim_{n\rightarrow\infty}\frac{\tilde{\varepsilon}_n(x)}{\varepsilon_n(x)}
  =1,
\end{equation}
uniformly for any $x\in\Omega$.
If $\rho$ is Lipschitz,
\begin{align*}
  \left|\rho(x)(\varepsilon_n(x))^d-\bar{\varepsilon}_n^d\right|
  =\left|\rho(x)(\varepsilon_n(x))^d-\frac{1}{\alpha_d}\mu\left(B(x,\varepsilon_n(x))\right) \right|
  \leq C (\varepsilon_n(x))^{d+1}.
\end{align*}
We further have
\begin{equation}\label{eq:th:4.1:2a}
  \lim_{n\rightarrow\infty}\frac{(\rho(x))^{1/d}\varepsilon_n(x)}
  {\bar{\varepsilon}_n}=1,
\end{equation}
uniformly for any $x\in\Omega$.

Let
\begin{align*}
  \chi_{k_n}(x_i,x_j)=
  \begin{cases}
    1,\quad \mbox{if } x_i\stackrel{k_n}{\sim} x_j,\\
    0,\quad \mbox{otherwise},
  \end{cases}
  \quad \forall x_i,x_j\in\Omega_n.
\end{align*}
We claim that
\begin{equation}\label{eq:th:4.1:3}
  \eta\left(\frac{x-y}{\hat{\varepsilon}_n(x)}\right)
  \leq
  \chi_{k_n}(T_n(y),T_n(x))
  \leq \eta\left(\frac{x-y}{\tilde{\varepsilon}_n(x)}\right),\quad \forall x, y\in\Omega.
\end{equation}
In fact, for any $x,y\in\Omega$,
\begin{equation*}
  T_n(y) \stackrel{k_n}{\sim} T_n(x),
\end{equation*}
yields
\begin{align*}
  \frac{k_n}{n}
  &\geq \mu_n\left(B(T_n(x), |T_n(x)-T_n(y)|)\right)\\
  &\geq
  \mu(B(T_n(x),|T_n(x)-T_n(y)|-\|T_n-Id\|_{L^\infty(\Omega)})) \\
  &\geq
  \mu(B(x,|T_n(x)-T_n(y)|-2\|T_n-Id\|_{L^\infty(\Omega)})),
\end{align*}
where we utilize $T_n(B(x,r-\|T_n-Id\|))\subset B(x,r)$ and
\begin{equation*}
  \mu_n(B(x,r))\geq \mu_n(T_n(B(x,r-\|T_n-Id\|)))=\mu(B(x,r-\|T_n-Id\|))
\end{equation*}
 for the second inequality.
This and \eqref{eq:th:4.1:1} imply that
\begin{equation*}
  |T_n(x)-T_n(y)|-2\|T_n-Id\|_{L^\infty(\Omega)}\leq \varepsilon_n(x).
\end{equation*}
Consequently,
\begin{equation*}
  |x-y|\leq |x-T_n(x)|+|y-T_n(y)|+|T_n(x)-T_n(y)|
  \leq \varepsilon_n(x) + 4\|T_n-Id\|_{L^\infty(\Omega)}=\tilde{\varepsilon}_n(x),
\end{equation*}
from which we obtain the second inequality in \eqref{eq:th:4.1:3}. The first inequality can be verified similarly.

2.
Let us now consider the liminf inequality. Assume that $(\mu_n,u_n)\rightarrow (\mu,u)$ in $TL^p(\Omega)$.
Estimate \eqref{eq:th:4.1:3} yields
\begin{align}\label{eq:th:4.1:4}
  \begin{split}
  &\mathcal{F}_{n,k_n}(u_n) \\
  &=\frac{1}{n\bar{\varepsilon}_n^{p}}\sum_{k=1}^{n}
  \max_{x_i,x_j\stackrel{k_n}{\sim} x_k}|u_n(x_i)-u_n(x_j)|^p \\
  &=\frac{1}{\bar{\varepsilon}_n^{p}}\int_\Omega\esssup_{y,z\in\Omega}
  \chi_{k_n}(T_n(z),T_n(x))\chi_{k_n}(T_n(y),T_n(x))
  |\tilde{u}_n(z)-\tilde{u}_n(y)|^p\rho(x)dx \\
  &\geq \frac{1}{\bar{\varepsilon}_n^{p}}\int_\Omega\esssup_{y,z\in\Omega}
  \eta\left(\frac{x-z}{\hat{\varepsilon}_n(x)}\right)
  \eta\left(\frac{x-y}{\hat{\varepsilon}_n(x)}\right)
  |\tilde{u}_n(z)-\tilde{u}_n(y)|^p\rho(x)dx \\
  &=\int_{\Omega}\left|\frac{(\rho(x))^{1/d}\hat{\varepsilon}_n(x)}{\bar{\varepsilon}_n}\right|^{p}\\
  &\qquad\qquad \frac{1}{|\hat{\varepsilon}_n(x)|^p}
  \esssup_{y,z\in\Omega}
  \eta\left(\frac{x-z}{\hat{\varepsilon}_n(x)}\right)
  \eta\left(\frac{x-y}{\hat{\varepsilon}_n(x)}\right)
  |\tilde{u}_n(z)-\tilde{u}_n(y)|^p(\rho(x))^{1-p/d}dx\\
  &\geq \inf_{x\in\Omega}\left|\frac{(\rho(x))^{1/d}\hat{\varepsilon}_n(x)}{\bar{\varepsilon}_n}\right|^{p} \\
  &\qquad \int_{\Omega}\frac{1}{|\hat{\varepsilon}_n(x)|^p}
  \esssup_{y,z\in\Omega}
  \eta\left(\frac{x-z}{\hat{\varepsilon}_n(x)}\right)
  \eta\left(\frac{x-y}{\hat{\varepsilon}_n(x)}\right)
  |\tilde{u}_n(z)-\tilde{u}_n(y)|^p(\rho(x))^{1-p/d}dx.
  \end{split}
\end{align}
It follows from \eqref{eq:th:4.1:2}--\eqref{eq:th:4.1:2a} and Lemma \ref{le:3.3} that
\begin{align*}
  \liminf_{n\rightarrow\infty} \mathcal{F}_{n,k_n}(u_n)
  &\geq \liminf_{n\rightarrow\infty}\inf_{x\in\Omega}\left|\frac{(\rho(x))^{1/d}\hat{\varepsilon}_n(x)}{\bar{\varepsilon}_n}\right|^{p}
  \mathcal{E}_{\hat{\varepsilon}_n}(\tilde{u}_n;\rho^{1-p/d})\\
  &\geq\mathcal{E}(u;\rho^{1-p/d}).
\end{align*}

3.
For the limsup inequality, we assume that $u\in W^{1,p}(\Omega)$ is Lipschitz. Let $u_n$ be the restriction of $u$ to the first $n$ data points. Then
\eqref{eq:th:4.1:3} implies that
\begin{align*}
  &\mathcal{F}_{n,k_n}(u_n)\\
  &\leq
  \frac{1}{|\bar{\varepsilon}_n|^p}\int_{\Omega}
  \esssup_{y,z\in\Omega}
  \eta\left(\frac{x-z}{\tilde{\varepsilon}_n(x)}\right)
  \eta\left(\frac{x-y}{\tilde{\varepsilon}_n(x)}\right)
  |\tilde{u}_n(z)-\tilde{u}_n(y)|^p\rho(x)dx\\
  &\leq
  \sup_{x\in\Omega}\left|\frac{(\rho(x))^{1/d}\tilde{\varepsilon}_n(x)}{\bar{\varepsilon}_n}\right|^{p} \\
  &\qquad \int_{\Omega}\frac{1}{|\tilde{\varepsilon}_n(x)|^p}
  \esssup_{y,z\in\Omega}
  \eta\left(\frac{x-z}{\tilde{\varepsilon}_n(x)}\right)
  \eta\left(\frac{x-y}{\tilde{\varepsilon}_n(x)}\right)
  |\tilde{u}_n(z)-\tilde{u}_n(y)|^p(\rho(x))^{1-p/d}dx\\
  &\leq
  \sup_{x\in\Omega}\left|\frac{(\rho(x))^{1/d}\tilde{\varepsilon}_n(x)}{\bar{\varepsilon}_n}\right|^{p} \\
  &\quad \int_{\Omega}\frac{1}{|\tilde{\varepsilon}_n(x)|^p}
  \esssup_{y,z\in\Omega}
  \eta\left(\frac{x-z}{\tilde{\varepsilon}_n(x)}\right)
  \eta\left(\frac{x-y}{\tilde{\varepsilon}_n(x)}\right)
  |{u}(z)-{u}(y)|^p(\rho(x))^{1-p/d}dx +I.\\
\end{align*}
It is not difficult to see that
\begin{equation*}
  |I|\leq  \frac{C}{\sup_{x\in\Omega}|\varepsilon_n(x)|^p}\|T_n-Id\|_{L^\infty(\Omega)}^p \rightarrow 0
\end{equation*}
as $n\rightarrow\infty$.
Consequently, by Lemma \ref{le:3.4},
\begin{align*}
  \limsup_{n\rightarrow\infty}\mathcal{F}_{n,k_n}(u_n)
  &\leq \limsup_{n\rightarrow\infty} \sup_{x\in\Omega}\left|\frac{(\rho(x))^{1/d}\tilde{\varepsilon}_n(x)}{\bar{\varepsilon}_n}\right|^{p}
  \mathcal{E}_{\tilde{\varepsilon}_n}(u;\rho^{1-p/d})\\
  &\leq\mathcal{E}(u;\rho^{1-p/d}).
\end{align*}

4.
To complete the proof for any continuous $\rho$, we approximate it by Lipschitz functions.
In fact, for the liminf inequality, there exists a sequence of Lipschitz functions $\{\rho_k(x)\}$ such that $\rho_k(x)\nearrow\rho(x)$ as $k\rightarrow\infty$ for any $x\in\Omega$.
Then
\begin{align*}
  \liminf_{n\rightarrow\infty}\mathcal{F}_{n,k_n}(u_n;\rho)
  \geq
  \lim_{k\rightarrow\infty}\liminf_{n\rightarrow\infty}\mathcal{F}_{n,k_n}(u_n;\rho_k)
  &= \lim_{k\rightarrow\infty}\mathcal{E}(u;\rho_k^{1-p/d})\\
  &=\mathcal{E}(u;\rho^{1-p/d}),
\end{align*}
where the penultimate equality comes from step 2 and the last equality comes from the monotone convergence theorem.

While for the liminf inequality, there exists a sequence of Lipschitz functions $\{\rho_k(x)\}$ such that $\rho_k(x)\searrow\rho(x)$ as $k\rightarrow\infty$ for any $x\in\Omega$.
Then
\begin{align*}
  \limsup_{n\rightarrow\infty}\mathcal{F}_{n,k_n}(u_n;\rho)
  \leq
  \lim_{k\rightarrow\infty}\limsup_{n\rightarrow\infty}\mathcal{F}_{n,k_n}(u_n;\rho_k)
  &= \lim_{k\rightarrow\infty}\mathcal{E}(u;\rho_k^{1-p/d})\\
  &=\mathcal{E}(u;\rho^{1-p/d}),
\end{align*}
where the penultimate equality comes from step 3 and the last equality comes from Lebesuge's dominated convergence theorem.
\end{proof}

The $\Gamma$-convergence of the constrained functionals follows from Theorem \ref{th:4.1}  directly.

\begin{theorem}\label{th:4.2}
  Let $p>d$ and $\delta_n\ll\bar{\varepsilon}_n\ll 1$. Then with probability one,
  \begin{equation*}
    \mathcal{F}_{n,k_n}^{con}\stackrel{\Gamma}{\longrightarrow }
    \mathcal{E}^{con}(~\cdot~ ;\rho^{1-p/d}),
  \end{equation*}
  in the $TL^p$ metric on the set $\{(\nu,g): \nu\in\mathcal{P}(\Omega), \|g\|_{L^\infty(\nu)}\leq M\}$, where $M= \max_{1\leq i\leq N}|y_i|$.
\end{theorem}

\begin{proof}
  Assume that
  \begin{equation}\label{eq:th:4.2:1}
    \mathcal{E}_{n,\varepsilon_n}(u_n)\leq C\mathcal{F}_{n,k_n}(u_n),
  \end{equation}
  for a constant $\varepsilon_n$ that satisfies $\delta_n\ll\varepsilon_n\ll 1$.
  Then the proof of the theorem is exactly the same as the proof of Theorem \ref{th:3.8}, except that we use \eqref{eq:th:4.2:1} to ensure that the conclusion of Lemma \ref{le:3.7} holds.

  To prove \eqref{eq:th:4.2:1}, we utilize \eqref{eq:th:4.1:4}. More precisely,
  \begin{align*}
    &\mathcal{F}_{n,k_n}(u_n) \\
    &\geq \inf_{x\in\Omega}\left|\frac{(\rho(x))^{1/d}\hat{\varepsilon}_n(x)}{\bar{\varepsilon}_n}\right|^{p} \\
    &\qquad \int_{\Omega}\frac{1}{|\hat{\varepsilon}_n(x)|^p}
    \esssup_{y,z\in\Omega}
    \eta\left(\frac{x-z}{\hat{\varepsilon}_n(x)}\right)
    \eta\left(\frac{x-y}{\hat{\varepsilon}_n(x)}\right)
    |\tilde{u}_n(z)-\tilde{u}_n(y)|^p(\rho(x))^{1-p/d}dx\\
    &\geq \inf_{x\in\Omega}\left|\frac{(\rho(x))^{1/d}\hat{\varepsilon}_n(x)}{\bar{\varepsilon}_n}\right|^{p}  \inf_{x\in\Omega}(\rho(x))^{-p/d}  \\
    &\qquad
    \frac{\varepsilon_n^p}{\sup_{x\in\Omega}|\hat{\varepsilon}_n(x)|^p}
    \int_{\Omega}\frac{1}{{\varepsilon}_n^p}
    \esssup_{y,z\in\Omega}
    \eta\left(\frac{x-z}{\tilde{\varepsilon}_n}\right)
    \eta\left(\frac{x-y}{\tilde{\varepsilon}_n}\right)
    |\tilde{u}_n(z)-\tilde{u}_n(y)|^p\rho(x)dx\\
    &\geq \frac{C}{{\varepsilon}_n^p}\int_{\Omega}
    \esssup_{y,z\in\Omega}
    \eta\left(\frac{T_n(x)-T_n(z)}{{\varepsilon}_n}\right)
    \eta\left(\frac{T_n(x)-T_n(y)}{{\varepsilon}_n}\right)
    |\tilde{u}_n(z)-\tilde{u}_n(y)|^p\rho(x)dx\\
    &=C \mathcal{E}_{n,\varepsilon_n}(u_n),
  \end{align*}
  where $\tilde{\varepsilon}_n=\inf_{x\in\Omega}\hat{\varepsilon}_n(x)$ and $\varepsilon_n=\tilde{\varepsilon}_n-2\|T_n-Id\|_{L^\infty(\Omega)}$.
  This completes the proof.
\end{proof}

The main result of this section is a corollary of Theorem \ref{th:4.2}. The proof is identical to that of Theorem \ref{th:3.1}.
\begin{theorem}\label{th:4.3}
  Let $p>d$ and $\delta_n\ll\bar{\varepsilon}_n\ll1$. If $u_n$ is a minimizer of $\mathcal{F}^{con}_{n,k_n}$, then almost surely,
  \begin{equation*}
    (\mu_n,u_n)\rightarrow (\mu,u),\quad\mbox{in } TL^p(\Omega),
  \end{equation*}
  as $n\rightarrow\infty$ for a function $u\in L^{p}(\Omega)$.
  Furthermore,
  \begin{equation*}
    \lim_{n\rightarrow\infty}\max_{x_k\in\Omega_n\cap\Omega'}|u_n(x_k)-u(x_k)|=0,
  \end{equation*}
  for any $\Omega'$ compactly contained in $\Omega$ and
  $u$ is a minimizer of $\mathcal{E}^{con}(~\cdot~ ;\rho^{1-p/d})$.
\end{theorem}

We conclude this section with two remarks.
\begin{remark}
In this paper, we consider the hypergraph $p$-Laplacian regularization with $p>1$ in a semisupervised setting. The results established in Section \ref{se:3} and Section \ref{se:4} can be generalized to $p=1$ in an unsupervised setting. The continuum limits of $\mathcal{E}_{n,\varepsilon_n}$ and $\mathcal{F}_{n,k_n}$ with $p=1$ become weighted total variations.
\end{remark}

\begin{remark}
Since the given data is a point cloud, we can compute the distance $|x_i-x_j|$ between any two vertices $x_i$ and $x_j$ and add this information to the objective function for a better performance.
A straightforward way is to construct a weight $w_{i,j}$ as a nonincreasing function of the distance $|x_i-x_j|$ and replace $|u(x_i)-u(x_j)|^p$ in both $\mathcal{E}_{n,\varepsilon_n}$ and $\mathcal{F}_{n,k_n}$ by $w_{i,j}|u(x_i)-u(x_j)|^p$.
All results in Section \ref{se:3} and Section \ref{se:4} can be generalized to the new energies without any essential modification under the assumption that
\begin{equation*}
  w_{i,j}=\eta(|x_i-x_j|),
\end{equation*}
and $\eta: [0,\infty)\rightarrow [0,\infty)$ is positive and continuous at $x=0$ and has compact support.
\end{remark}

\section{The numerical algorithm}

In this section, we present an algorithm for solving the optimization problem
  \begin{align}\label{eq:5.0}
    \min_{u\in\mathbb{R}^n}\frac{1}{p}\sum_{k=1}^n\max_{x_i,x_j\in e_k}w_{i,j}\left|u(x_i)-u(x_j)\right|^p
    \quad \mbox{s.t. } u(x_i)=y_i \mbox{ for } x_i\in\mathcal{O},
  \end{align}
  where $e_k=\{x_j\in\Omega_n: |x_k-x_j|\leq \varepsilon_n\}$ for the $\varepsilon_n$-ball hypergraph or  $e_k=\{x_j\in\Omega_n: x_j\stackrel{k_n}{\sim} x_k \}$ for the $k_n$-NN hypergraph.

Let
\begin{align*}
  \mathbb{I}_{\mathcal{O}}(u)=
  \begin{cases}
    0,\quad \mbox{if } u(x_i)=y_i \mbox{ for any } x_i\in\mathcal{O},\\
    \infty,\quad \mbox{otherwise},
  \end{cases}
\end{align*}
be an indicator function. Then problem \eqref{eq:5.0} can be rewritten as an
unconstrained problem
\begin{equation}\label{eq:5.1}
  \min_{u\in\mathbb{R}^n}\sum_{k=1}^ng\left(A_ku\right) + \mathbb{I}_{\mathcal{O}}(u),
\end{equation}
  where
\begin{equation*}
  g(\beta)=\frac{1}{p}\left(\max|\beta|\right)^p, \quad \beta\in\mathbb{R}^m,
\end{equation*}
$m=\max\left\{\frac{|e_k|(|e_k|-1)}{2}: k=1,\cdots,n\right\}$, and $A_k\in\mathbb{R}^{m\times n}$ such that
  \begin{equation*}
    g\left(A_ku\right)=\frac{1}{p}\max_{x_i,x_j\in e_k}\left|w_{i,j}^{\frac{1}{p}}(u(x_i)-u(x_j))\right|^p.
  \end{equation*}
Clearly, both $g$ and $\mathbb{I}_{\mathcal{O}}$ are proper,  lower semicontinuous, and convex functions.
A popular algorithm for solving \eqref{eq:5.1} is the stochastic primal-dual hybrid gradient (SPDHG) algorithm \cite{chambolle2018stochastic}
that rewrites the original optimal problem as a saddle point problem and updates the primal variable and the dual variable separately.

By introducing dual variables $\alpha_i\in\mathbb{R}^m$, $i=1,\cdots,n$, problem \eqref{eq:5.1} is equivalent to the saddle point problem
\begin{equation}\label{eq:saddle}
  \min_{u\in\mathbb{R}^n}\max_{\substack{\alpha_i\in\mathbb{R}^m\\i=1,\cdots,n}}\sum_{i=1}^n\langle A_iu,\alpha_i\rangle -g^*(\alpha_i)+\mathbb{I}_{\mathcal{O}}(u),
\end{equation}
where
\begin{equation*}
  g^*(\alpha_i)=\max_{\beta\in \mathbb{R}^m}\langle \beta,\alpha_i\rangle -g(\beta),\quad i=1,2,\cdots,n,
\end{equation*}
denotes the Fenchel conjugate of $g$ and we utilize the property $g^{**}=g$.
Now SPDHG with serial sampling is summarized in Algorithm \ref{alg1}.
Under the condition
\begin{equation*}
  \sigma\tau<\max_{i=1,\cdots,n}\frac{p_i}{\|A_i\|^2},
\end{equation*}
it converges almost surely to a solution of problem \eqref{eq:5.1} \cite{alacaoglu2022convergence}.

\begin{algorithm}
  \caption{SPDHG for solving problem \eqref{eq:saddle}.}
  \label{alg1}
  \begin{algorithmic}
  \REQUIRE Primal variable $u^{(0)}$, dual variables $\alpha_i^{(0)}$, $i=1,2,\cdots,n$; step sizes  $\tau, \sigma$; probabilities $p_i$, $i=1,\cdots,n$.
  \STATE Initialization: $\bar{\alpha}^{(0)}_i=\alpha^{(0)}_i$, $i=1,\cdots,n$; $k=0$.
  \WHILE { the stopping criterion is not satisfied}
  \STATE Update the primal variable:
  \begin{align}\label{eq:5.2}
  \begin{split}
    u^{(k+1)}=&\mbox{prox}_{\tau\mathbb{I}_{\mathcal{O}}} \left(u^{(k)}-\tau\sum_{i=1}^nA_i^T\bar{\alpha}_i^{(k)}\right)\\
    :=&\arg\min_{u\in\mathbb{R}^n} \tau\mathbb{I}_{\mathcal{O}}(u) + \frac{1}{2}\left\|u-\left(u^{(k)}-\tau\sum_{i=1}^nA_i^T\bar{\alpha}_i^{(k)}\right)\right\|_2^2.
  \end{split}
  \end{align}
  \STATE Update dual variables: Randomly pick $i\in \{1,2,\cdots,n\}$ with probability $p_i$,
  \begin{equation}\label{eq:5.3}
    \alpha^{(k+1)}_j=
    \begin{cases}
      \mbox{prox}_{\sigma g^*}\left(\alpha^{(k)}_i+\sigma A_iu^{(k+1)}\right), \quad &\mbox{if } j=i,\\
      \alpha^{(k)}_j,\quad &\mbox{if } j\neq i.
    \end{cases}
  \end{equation}
  \STATE Extrapolation on dual variables:
  \begin{equation*}
    \bar{\alpha}^{(k+1)}_j=
    \begin{cases}
      \alpha^{(k+1)}_i+\frac{1}{p_i} \left(\alpha^{(k+1)}_i-\alpha^{(k)}_i\right),\quad &\mbox{if } j=i,\\
      \alpha^{(k)}_j,\quad &\mbox{if } j\neq i.
    \end{cases}
  \end{equation*}
  \ENDWHILE
  \RETURN $u^{(k+1)}$.
  \end{algorithmic}
  \end{algorithm}

Subproblem \eqref{eq:5.2} has a closed form solution
\begin{align*}
  u^{(k+1)}(x_j)=
  \begin{cases}
    y_j,\quad &\mbox{if } x_j\in\mathcal{O},\\
    \left(u^{(k)}-\tau\sum_{i=1}^nA_i^T\bar{\alpha}_i^{(k)}\right)(x_j),\quad&\mbox{otherwise.}
  \end{cases}
\end{align*}
The rest is to find the Fenchel conjugate $g^*$ and solve subproblem
\eqref{eq:5.3}.


\begin{lemma}
Let
\begin{equation*}
  h(t)=\frac{1}{p}|t|^p,\quad t\in\mathbb{R}.
\end{equation*}
Then
\begin{align}\label{eq:5.5a}
  h^*(s)=
    \begin{cases}
      0,\quad \mbox{if } |s|\leq 1,\\
      \infty,\quad \mbox{otherwise,}
    \end{cases}
    \mbox{if } p=1,
\end{align}
or
\begin{equation}\label{eq:5.5b}
  h^*(s)=\frac{1}{p'}|s|^{p'}, \quad \mbox{if } p>1,
\end{equation}
where $p'=\frac{p}{p-1}$.
Moreover,
\begin{equation}\label{eq:5.6}
  g^*(\alpha)=h^*(\|\alpha\|_1),\quad \alpha\in\mathbb{R}^m.
\end{equation}
\end{lemma}

\begin{proof}
 By the definition of the Fenchel conjugate, \eqref{eq:5.5a} and \eqref{eq:5.5b} can be verified directly. We prove \eqref{eq:5.6} in the following.

For any $\alpha, \beta\in \mathbb{R}^m$,
\begin{align*}
   \langle \alpha,\beta\rangle -g(\beta)\leq \|\alpha\|_1 \max |\beta|-\frac{1}{p}\left(\max |\beta|\right)^p\leq h^*(\|\alpha\|_1),
\end{align*}
from which we obtain that $g^*(\alpha)\leq h^*(\|\alpha\|_1)$.
On the other hand,
for any $\alpha\in \mathbb{R}^m$, let $\gamma\in \mathbb{R}^m$ such that $\max |\gamma|=1$ and $\langle \alpha,\gamma\rangle=\|\alpha\|_1$.
Then
\begin{align*}
  g^*(\alpha)\geq \max_{\substack{\beta=t\gamma\\t\in\mathbb{R}}}\langle \alpha,\beta\rangle -g(\beta)
  =\max_{t\in\mathbb{R}}\|\alpha\|_1 t-\frac{1}{p}|t|^p=h^*(\|\alpha\|_1).
\end{align*}
This completes the proof.
\end{proof}

Now if $p=1$, subproblem \eqref{eq:5.3} with $j=i$ becomes a projection onto the $L^1$ ball of the following form
\begin{equation*}
  \alpha^{(k+1)}_i=\arg\min_{\substack{\alpha\in\mathbb{R}^m\\ \|\alpha\|_1\leq 1}}\left\|\alpha-\left(\alpha^{(k)}_i+\sigma A_iu^{(k+1)}\right)\right\|_2^2,
\end{equation*}
which can be solved efficiently \cite{condat2016fast}.
In the case of $p>1$, the solution is given as follows.
\begin{lemma}
Let $p>1$.
  The solution of subproblem \eqref{eq:5.3} with $j=i$ satisfies
  \begin{equation}\label{eq:5.8}
    \alpha^{(k+1)}_i=\mbox{sign}(\alpha^{(k)}_i+\sigma A_iu^{(k+1)})\max\left\{\left|\alpha^{(k)}_i+\sigma A_iu^{(k+1)}\right|-\sigma \|\alpha^{(k+1)}_i\|_1^{p'-1}, 0\right\}.
  \end{equation}
\end{lemma}
All operations involving vectors in \eqref{eq:5.8} are understood elementwise.

\begin{proof}
For simplicity of notation, let $\beta=\alpha^{(k)}_i+\sigma A_iu^{(k+1)}$ and $\alpha^*=\alpha^{(k+1)}_i$. Then subproblem  \eqref{eq:5.3} becomes
\begin{equation*}
  \alpha^*=\arg\min_{\alpha\in\mathbb{R}^m}\frac{1}{2}\|\alpha-\beta\|_2^2+\frac{\sigma}{p'}\left(\|\alpha\|_1\right)^{p'},\quad p>1.
\end{equation*}
Since the right-hand side is non-separable, we introduce a new variable $t=\|\alpha\|_1$ and consider the Lagrangian
\begin{equation*}
  \frac{1}{2}\|\alpha-\beta\|_2^2+\frac{\sigma}{p'}t^{p'}+\lambda(\|\alpha\|_1-t).
\end{equation*}
The saddle point $(\alpha^*, \lambda^*, t^*)$ satisfies
\begin{align*}
  \begin{cases}
  \alpha^*\in\arg\min_{\alpha\in\mathbb{R}^m}  \frac{1}{2}\|\alpha-\beta\|_2^2+\lambda^*\|\alpha\|_1,\\
  \lambda^*=\sigma (t^*)^{p'-1},\\
  t^*=\|\alpha^*\|_1.
  \end{cases}
\end{align*}
The $L^1$ regularized problem has a closed-form solution
$\alpha^*=\mbox{sign}(\beta)\max\{|\beta|-\lambda^*, 0\}$. By substituting
$\lambda^*=\sigma (t^*)^{p'-1}=\sigma \|\alpha^*\|_1^{p'-1}$ into it, we arrive at
\begin{equation}\label{eq:5.9}
  \alpha^*=\mbox{sign}(\beta)\max\{|\beta|-\sigma \|\alpha^*\|_1^{p'-1}, 0\},
\end{equation}
which is exactly \eqref{eq:5.8} and is the solution of subproblem \eqref{eq:5.3}.
\end{proof}

The last step is to solve $\alpha^{(k+1)}_i$ from \eqref{eq:5.8}, which can be done in a few iterations.
For the rest of this section, the subscript $j$ of a vector denotes the $j$-th element of it.
To explain the basic idea, we continue to use the notation of \eqref{eq:5.9} for simplicity.

It follows from \eqref{eq:5.9} that $\alpha^*=\mbox{sign}(\beta)|\alpha^*|$ and $|\alpha^*|$ depends only on $|\beta|$.
W.l.o.g., we assume that $\beta\geq 0$ and $\beta$ is ordered in a nondecreasing order.
As a result, $\alpha^*\geq 0$ and $\alpha^*$ is also in a nondecreasing order.

\begin{lemma}\label{le:5.3}
  Let $\beta\in\mathbb{R}^m$ be a nonnegative vector in a nondecreasing order and $\beta\neq 0$. If $\sigma>0$ and $p>1$,
  there exists a unique vector $\alpha\in\mathbb{R}^m$ such that
  \begin{equation}\label{eq:5.10}
    \alpha=\beta-\sigma\left(\sum_{j=1}^m\alpha_j\right)^{p'-1},
  \end{equation}
  and
  \begin{equation*}
    \lambda:=\sigma\left(\sum_{j=1}^m\alpha_j\right)^{p'-1}> 0.
  \end{equation*}
  If $\alpha\geq 0$, then $\alpha^*=\alpha$, where $\alpha^*$ is the solution of equation \eqref{eq:5.9}.
  Otherwise, we have
  \begin{equation}\label{eq:5.10a}
    \lambda\leq\lambda^*:=\sigma \|\alpha^*\|_1^{p'-1}.
  \end{equation}
\end{lemma}
\begin{proof}
  It follows from \eqref{eq:5.10} that
  \begin{equation}\label{eq:5.11}
    s+m\sigma s^{p'-1}-\sum_{j=1}^m\beta_j=0,
  \end{equation}
  where $s=\sum_{j=1}^m\alpha_j$.
  Since $s+m\sigma s^{p'-1}$ is an increasing function on $\mathbb{R}^+$ and $\sum_{j=1}^m\beta_j> 0$, we obtain that equation \eqref{eq:5.11} admits a unique solution $s>0$.
  By substituting it into \eqref{eq:5.10}, we obtain a solution $\alpha$ with $\lambda> 0$.

  If $\alpha\geq 0$, then $\alpha$ also satisfies equation \eqref{eq:5.9} and $\alpha^*=\alpha$. Otherwise, there exists at least one zero element in $\alpha^*$. Assume that
  \begin{equation*}
    0=\alpha^*_1=\cdots=\alpha^*_k< \alpha^*_{k+1}\leq\cdots\leq \alpha^*_{m}.
  \end{equation*}
Then
  \eqref{eq:5.9} yields
  \begin{equation}\label{eq:5.12a}
    \beta_k\leq \lambda^*\leq \beta_{k+1},
  \end{equation}
  and
  \begin{equation}\label{eq:5.12}
    s^*+(m-k)\sigma s^{*p'-1}-\sum_{j=k+1}^m\beta_j=0,
  \end{equation}
  where $s^*=\|\alpha^*\|_1$.
  Now \eqref{eq:5.11}--\eqref{eq:5.12} imply that
  \begin{align*}
    m\lambda&=m\sigma s^{p'-1}=\sum_{j=1}^m\beta_j-s
    =\sum_{j=k+1}^m\beta_j+\sum_{j=1}^k\beta_j-s \\
    &=s^*+(m-k)\sigma s^{*p'-1}+\sum_{j=1}^k\beta_j-s\\
    &=(m-k)\lambda^*+\sum_{j=1}^k\beta_j +s^*-s \\
    &\leq m\lambda^* + \left(\frac{\lambda^*}{\sigma}\right)^{\frac{1}{p'-1}}-\left(\frac{\lambda}{\sigma}\right)^{\frac{1}{p'-1}},
  \end{align*}
  or equivalently,
  \begin{equation}
    m\lambda+\left(\frac{\lambda}{\sigma}\right)^{p-1}\leq m\lambda^* + \left(\frac{\lambda^*}{\sigma}\right)^{p-1}.
  \end{equation}
  From the monotonicity of $m\lambda+\left(\frac{\lambda}{\sigma}\right)^{p-1}$, we conclude that $\lambda\leq \lambda^*$.
\end{proof}

Lemma \ref{le:5.3} provides an algorithm for solving \eqref{eq:5.8}.
The key idea is that
\eqref{eq:5.10a} implies that $\alpha^*_j=0$ for any $j\in \{j: \alpha_j<0\}$.
Consequently, we can use Lemma \ref{le:5.3} again to solve the remaining elements of $\alpha^*$.
It should be noted that the proof of Lemma \ref{le:5.3} does not rely on the monotonicity of $\beta$.
The details are summarized in Algorithm \ref{alg2}.

\begin{algorithm}
  \caption{Solving $\mbox{prox}_{\sigma g^*}(\beta)$ for subproblem \eqref{eq:5.3} with $p>1$.}
  \label{alg2}
  \begin{algorithmic}
  \REQUIRE $\beta\in\mathbb{R}^m$; $\sigma>0$, $p>1$.
  \STATE Initialization: $\beta^+=|\beta|$, $\alpha=\beta^+$; stopping criterion $c=-1$.
  \WHILE { $c<0$}
  \STATE Find positive elements: $\gamma=\{\beta^+_j: j\in I\}\in\mathbb{R}^b$, where $I=\{j: \alpha_j>0\}$.
  \STATE Solve $s+b\sigma s^{p'-1}-\|\gamma\|_1=0$ on $\mathbb{R}^+$.
  \STATE Obtain the threshold: $\lambda=\sigma s^{p'-1}$.
  \STATE Update the solution: $\alpha=\max\{\beta^+-\lambda, 0\}$.
  \STATE $c=\min\gamma-\lambda$.
  \ENDWHILE
  \STATE Fix the negative part: $\alpha_j=-\alpha_j$ for any $j\in \{j: \beta_j<0\}$.
  \RETURN $\alpha$.
  \end{algorithmic}
\end{algorithm}

\section{Experimental results}
In this section, we present numerical experiments of the hypergraph $p$-Laplacian regularization (HpL) for data interpolation.
The purpose is to
verify the theoretical results established in previous sections and
make comparisons with the graph $p$-Laplacian regularization (GpL).
See \eqref{eq:GpL} and \eqref{eq:5.0} for the definition of GpL and HpL.

We start with experiments in 1D and then discuss performance of GpL and HpL on some higher-dimensional data interpolation problems, including semi-supervised learning and image inpainting.

\subsection{Experiments in 1D}
Let $d=1$ and $n=1280$.  We select $n$ random numbers under the standard uniform distribution on the interval $(0,1)$ and let $\Omega_n$ be the set of random numbers.
The set of labeled points $\mathcal{O}$ consists of 6 points in $\Omega_n$,
which are marked by red circles in Figure \ref{fig:1Deps}--\ref{fig:1Dp}.
We shall discuss the difference between GpL and HpL for the interpolation of signals defined on $\Omega_n$ with given labeled points $\mathcal{O}$. The weight function $w_{i,j}=1$ is used for both GpL and HpL.

Let us first consider the case of the $\varepsilon_n$-ball graph/hypergraph and $p=2$.
According to \cite{slepcev2019analysis}, the minimizer of GpL with the constraint on $\mathcal{O}$ is an approximation of the minimizer of the continuum $p$-Laplacian regularization with the constraint on $\mathcal{O}$ (the minimizer is a piecewise linear and continuous function that passes through the labeled points) for moderate connection radius $\varepsilon_n$ (i.e., $\delta_n\ll\varepsilon_n\ll\left(\frac{1}{n}\right)^{1/p}$).
This is verified by the results shown in the first row of Figure \ref{fig:1Deps}.
We consider the case of $\varepsilon_n\geq 0.006$ since the graph becomes disconnected when $\varepsilon_n\leq 0.005$.
When $\varepsilon_n= 0.006$, GpL gives an approximation of the continuous solution.
With an increase in $\varepsilon_n$, the error of GpL increases.
It develops spikes at the labeled points for large $\varepsilon_n$.

The second row of Figure \ref{fig:1Deps}, which shows the results of HpL with different $\varepsilon_n$, is quite different from the first row.
As the connection radius $\varepsilon_n$ increases, the error of HpL also increases.
However, it alleviates the development of spikes even for large $\varepsilon_n$ (see the 3rd and the 4th labeled points).
This partially verifies the conclusion of Theorem \ref{th:3.1}, showing that the convergence of the hypergraph $p$-Laplacian regularization to the continuum $p$-Laplacian regularization in a semisupervised setting requires a weaker assumption on the upper bound of $\varepsilon_n$ (i.e., $\delta_n\ll\varepsilon_n\ll 1$).
Two spikes (the 2nd and the 5th labeled points) appear in the results because the assumption $\varepsilon_n\ll 1$ is no longer fulfilled.

A similar conclusion can be obtained for the $k_n$-NN graph/hypergraph.
In Figure \ref{fig:1Dk}, we present the results of GpL and HpL with different $k_n$. Again we observe that GpL develops spikes at the labeled points for large $k_n$ while HpL performs better.
This coincides with the result we established in Theorem \ref{th:4.3}.

Theoretically, the minimizer of HpL gains more smoothness as the power $p$ increases.
Figure \ref{fig:1Dp} shows the results of HpL with different $p$.
It is observed that a higher $p$ also prevents the occurrence of spikes.
A similar conclusion has been obtained for GpL \cite{el2016asymptotic}.
In the case of $p=1$ where the assumption $p>d$ is not fulfilled, HpL gives noninformative solutions.

\begin{figure}[htbp]
  \centering
  \begin{tabular}{@{}c@{~}c@{~}c@{~}c}
    \includegraphics[width=.24\textwidth]{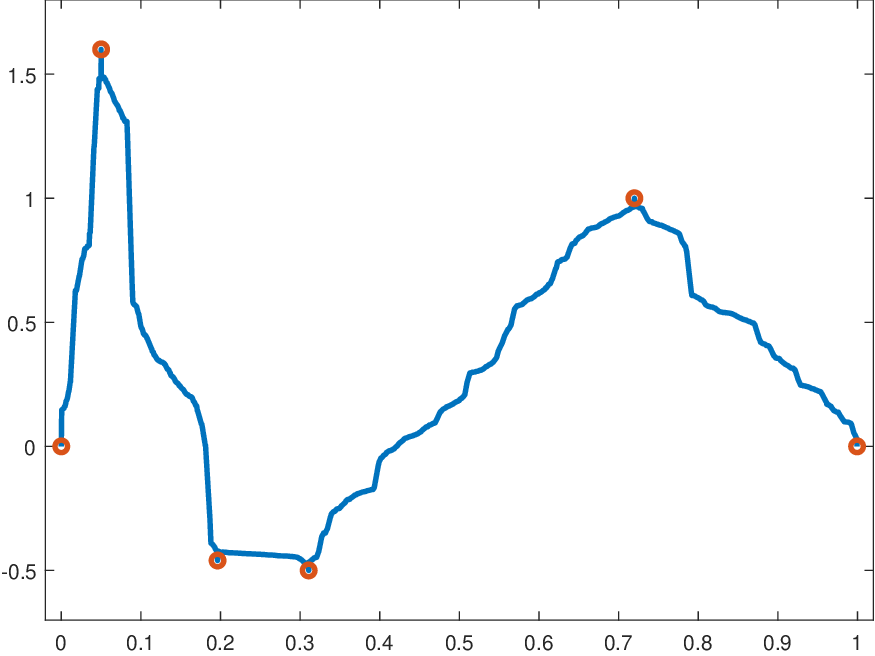} &
    \includegraphics[width=.24\textwidth]{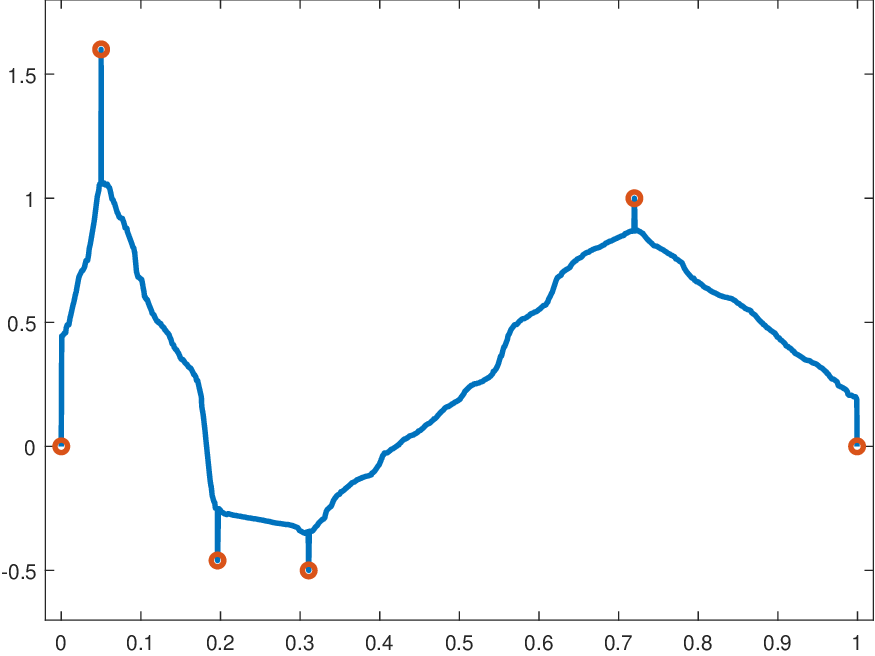} &
    \includegraphics[width=.24\textwidth]{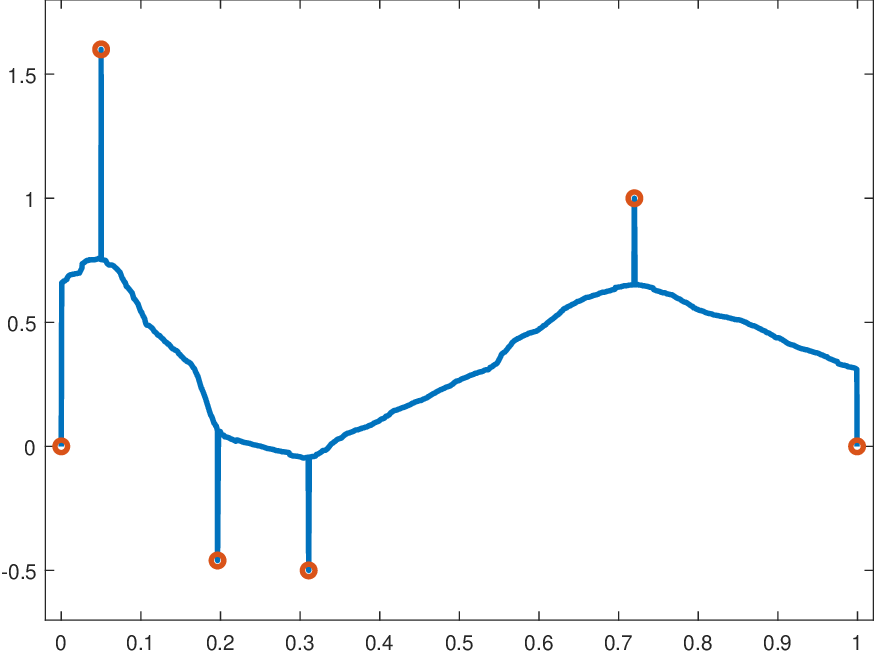} &
      \includegraphics[width=.24\textwidth]{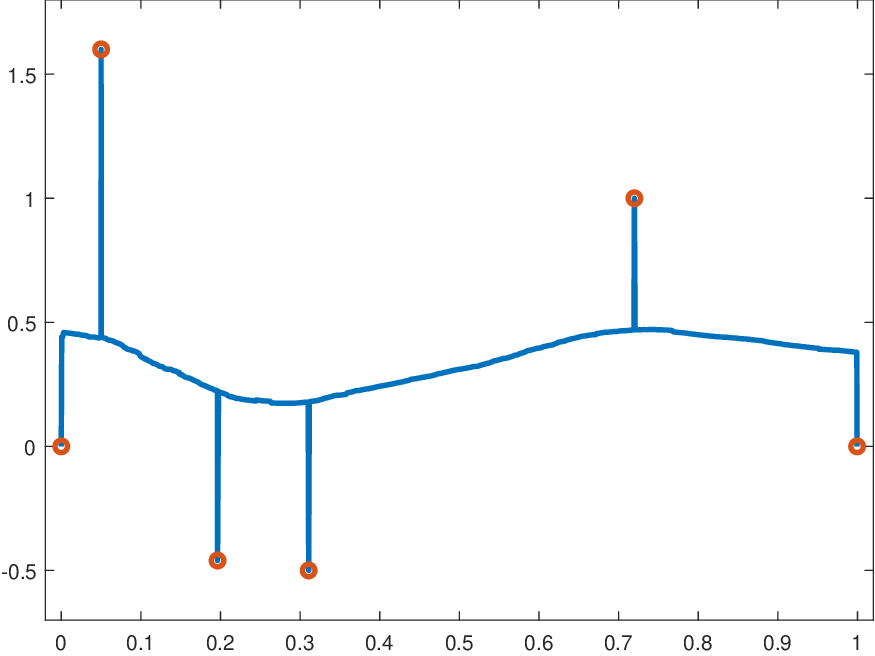} \\
      \includegraphics[width=.24\textwidth]{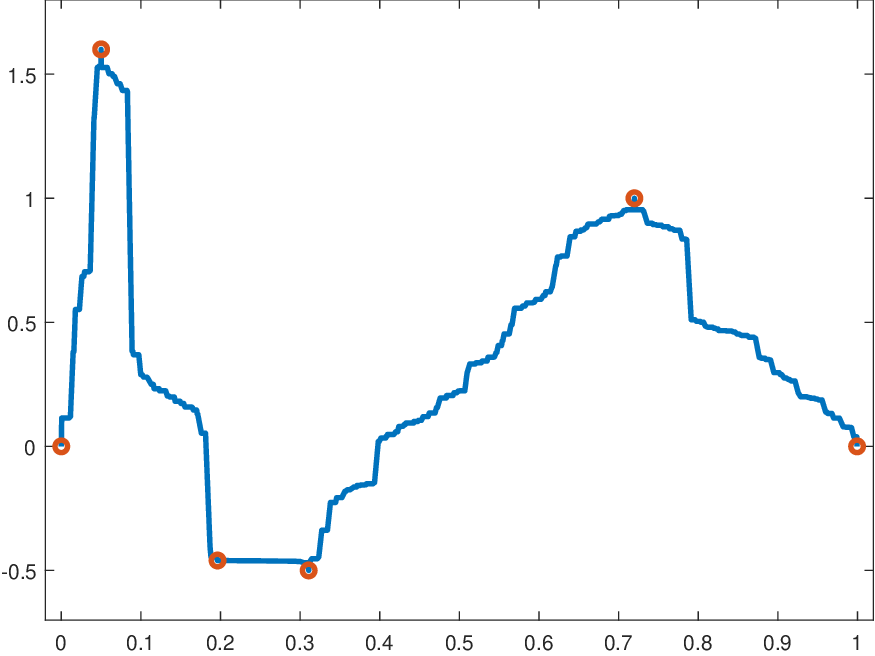} &
      \includegraphics[width=.24\textwidth]{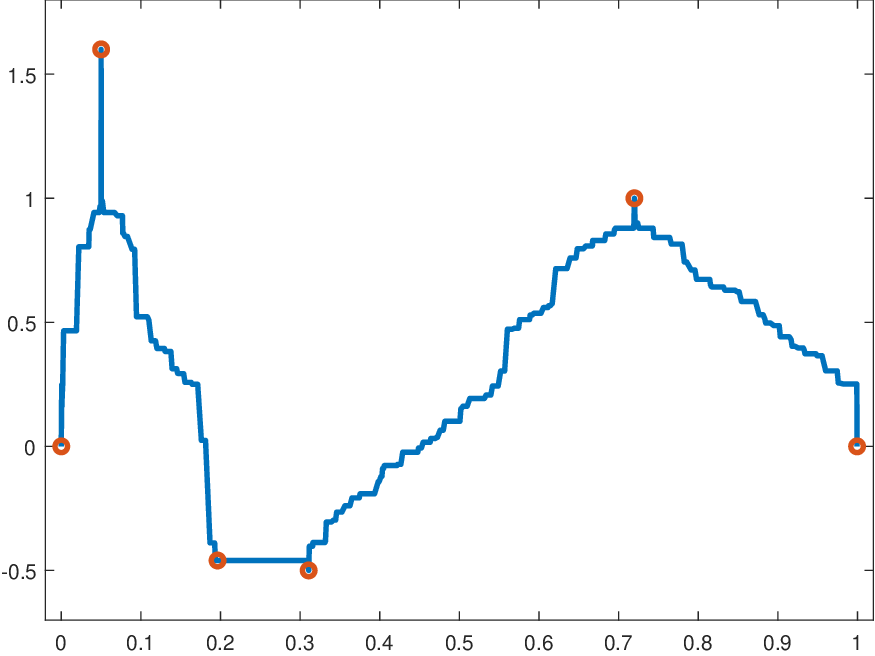} &
      \includegraphics[width=.24\textwidth]{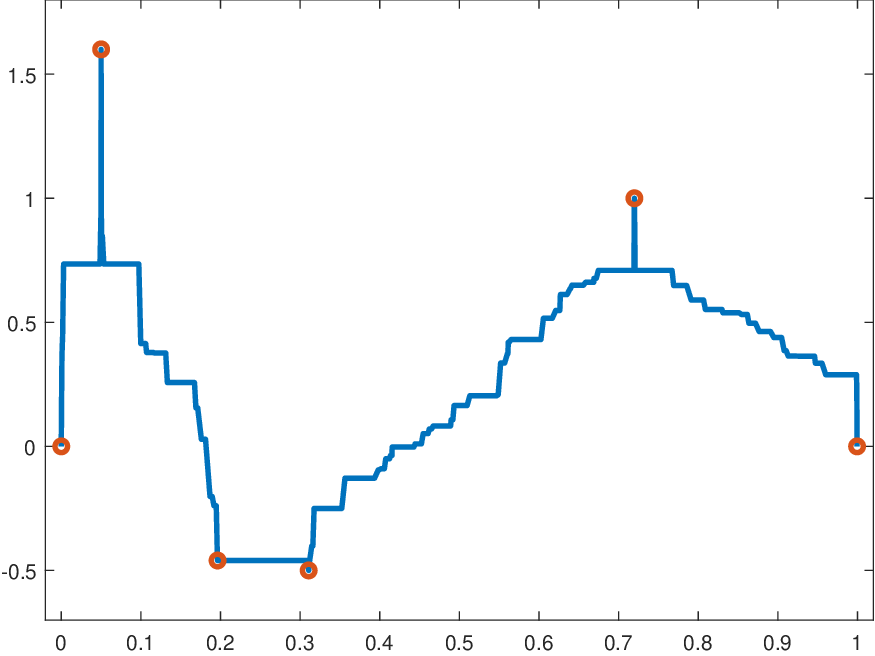} &
      \includegraphics[width=.24\textwidth]{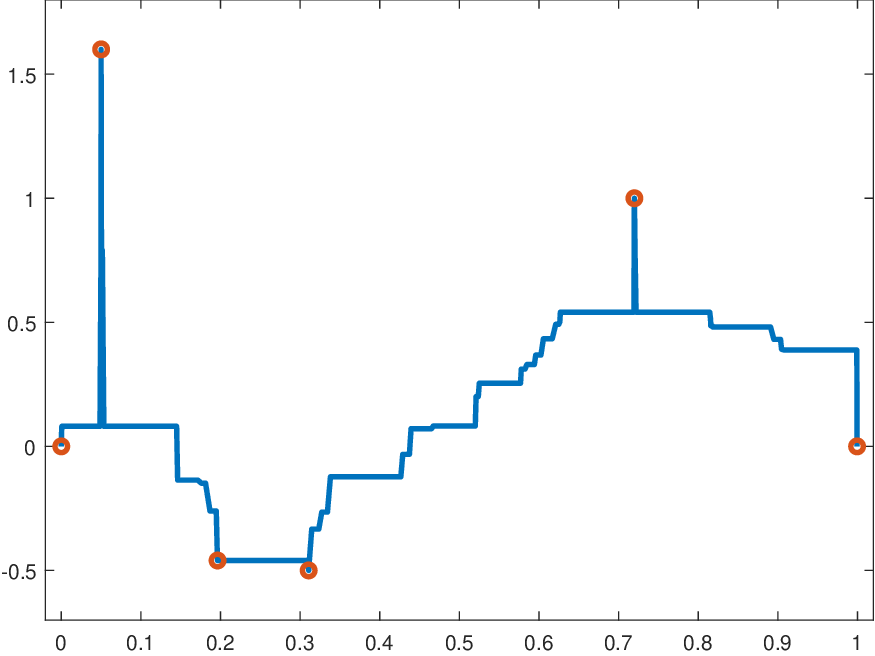}  \\
      \footnotesize\emph{$\varepsilon_n=0.006$}  &
      \footnotesize\emph{$\varepsilon_n=0.012$} &
      \footnotesize\emph{$\varepsilon_n=0.024$} &
      \footnotesize\emph{$\varepsilon_n=0.048$} \\
    \end{tabular}
    \caption{Results of GpL and HpL with $p=2$ and different $\varepsilon_n$. First row: GpL; second row: HpL.}
    \label{fig:1Deps}
  \end{figure}

  \begin{figure}[htbp]
    \centering
    \begin{tabular}{@{}c@{~}c@{~}c@{~}c}
      \includegraphics[width=.24\textwidth]{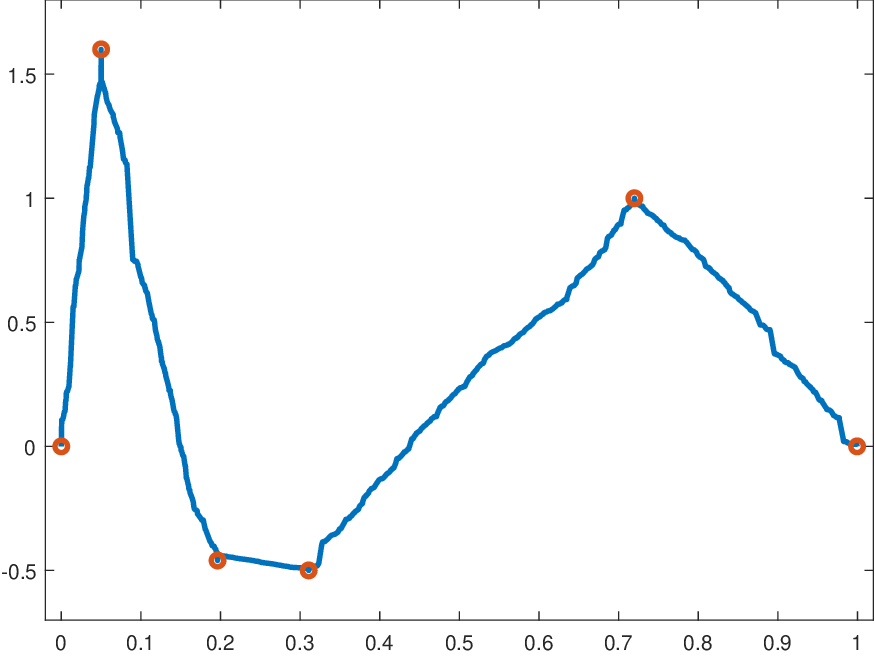} &
      \includegraphics[width=.24\textwidth]{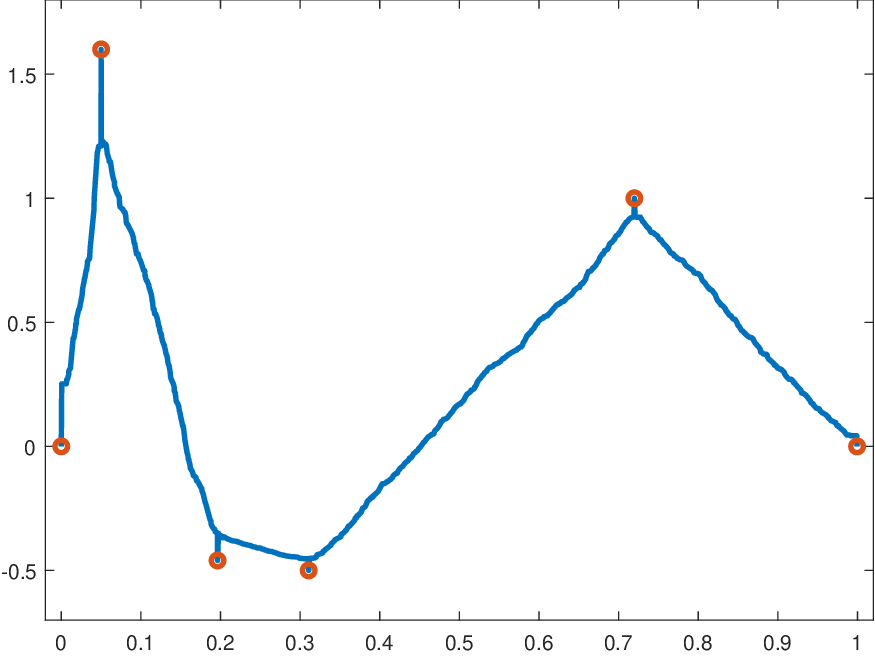} &
      \includegraphics[width=.24\textwidth]{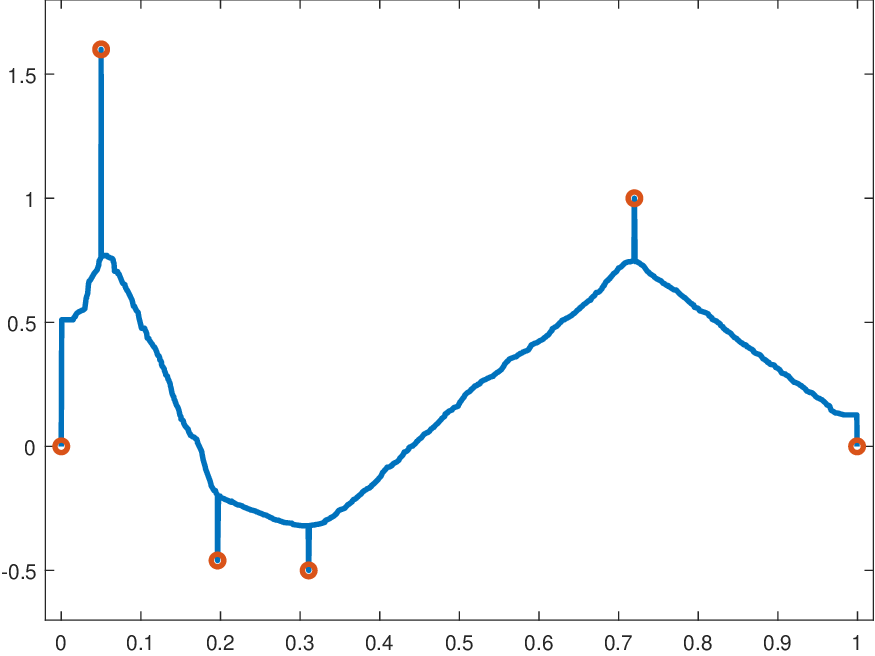} &
        \includegraphics[width=.24\textwidth]{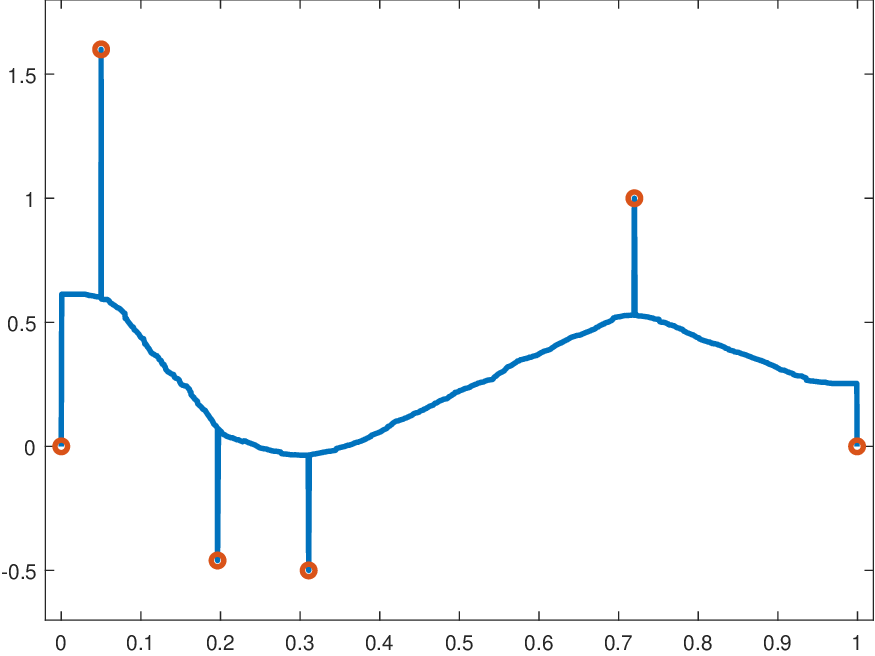} \\
        \includegraphics[width=.24\textwidth]{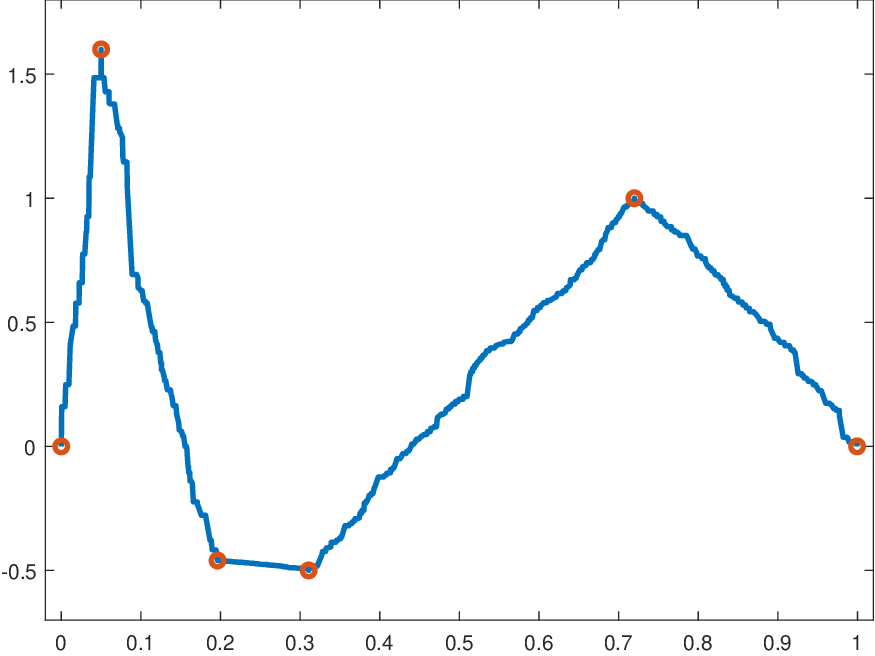} &
        \includegraphics[width=.24\textwidth]{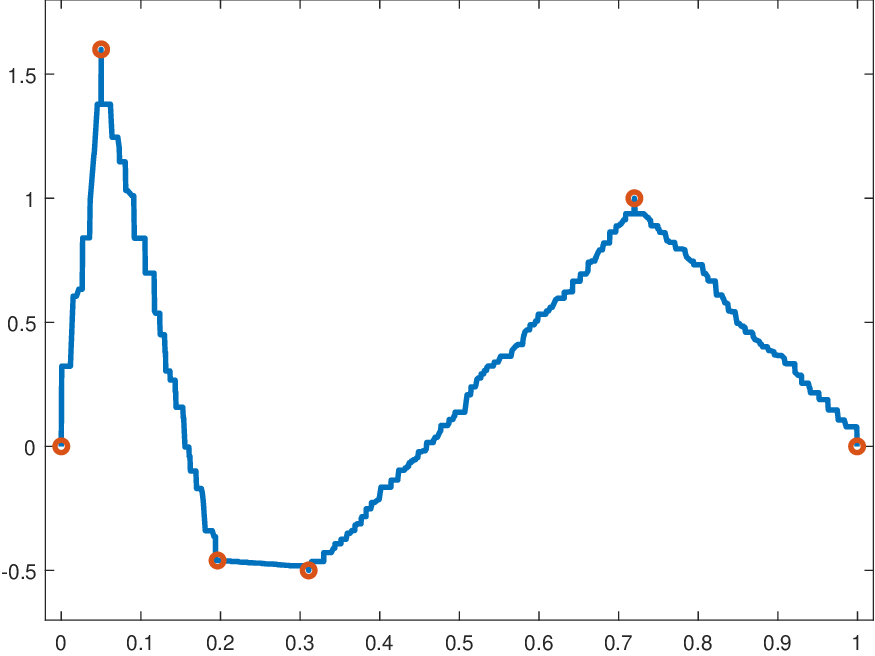} &
        \includegraphics[width=.24\textwidth]{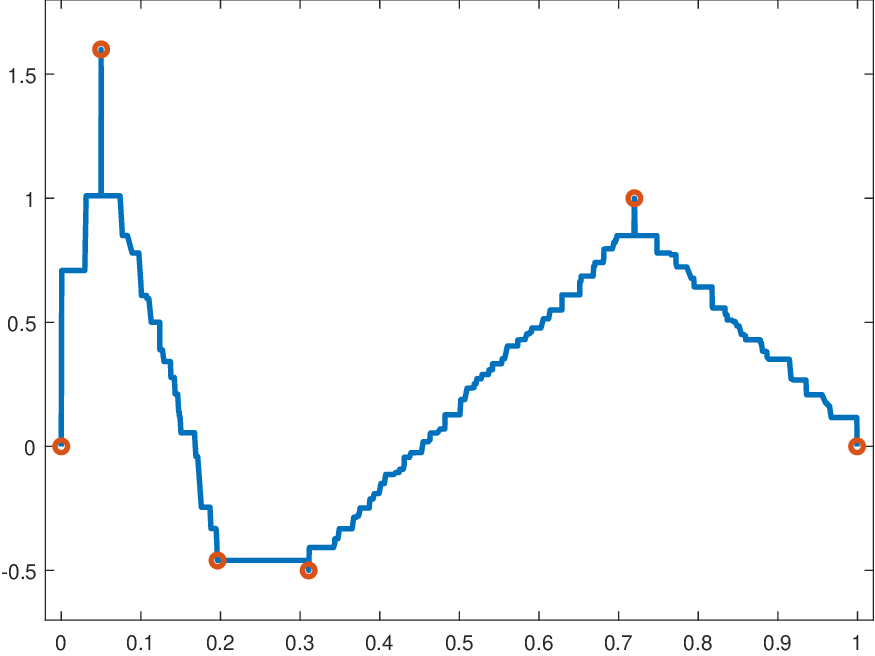} &
        \includegraphics[width=.24\textwidth]{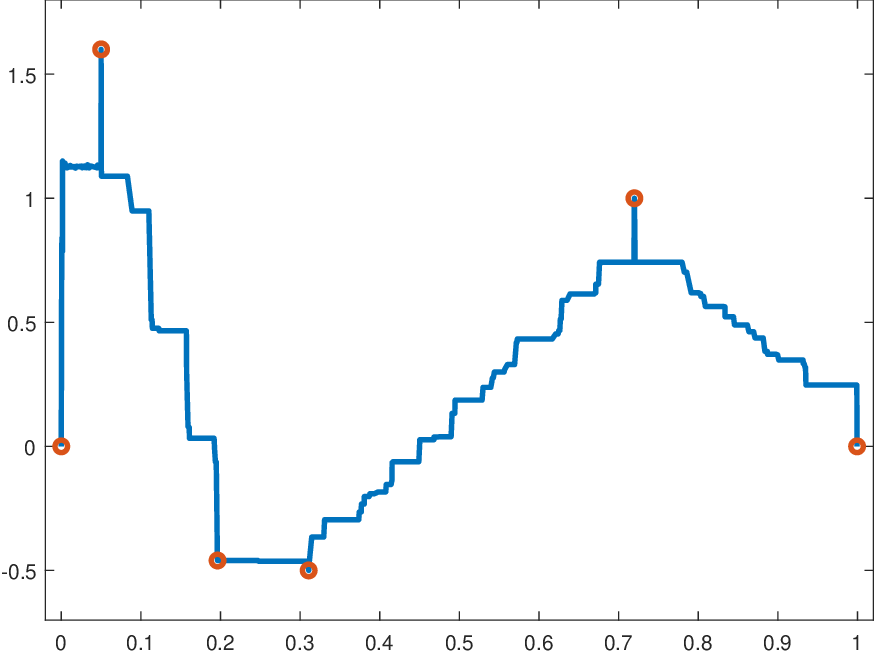}  \\
        \footnotesize\emph{$k_n=9$}  &
        \footnotesize\emph{$k_n=18$} &
        \footnotesize\emph{$k_n=36$} &
        \footnotesize\emph{$k_n=72$} \\
      \end{tabular}
      \caption{Results of GpL and HpL with $p=2$ and different $k_n$. First row: GpL; second row: HpL.}
      \label{fig:1Dk}
    \end{figure}

    \begin{figure}[htbp]
      \centering
      \begin{tabular}{@{}c@{~}c@{~}c@{~}c}
        \includegraphics[width=.24\textwidth]{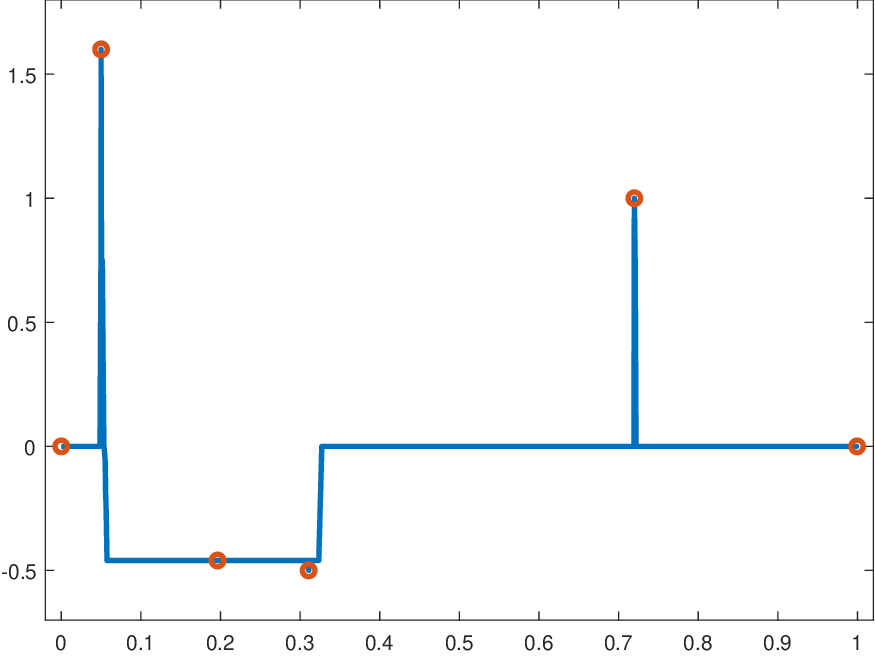} &
        \includegraphics[width=.24\textwidth]{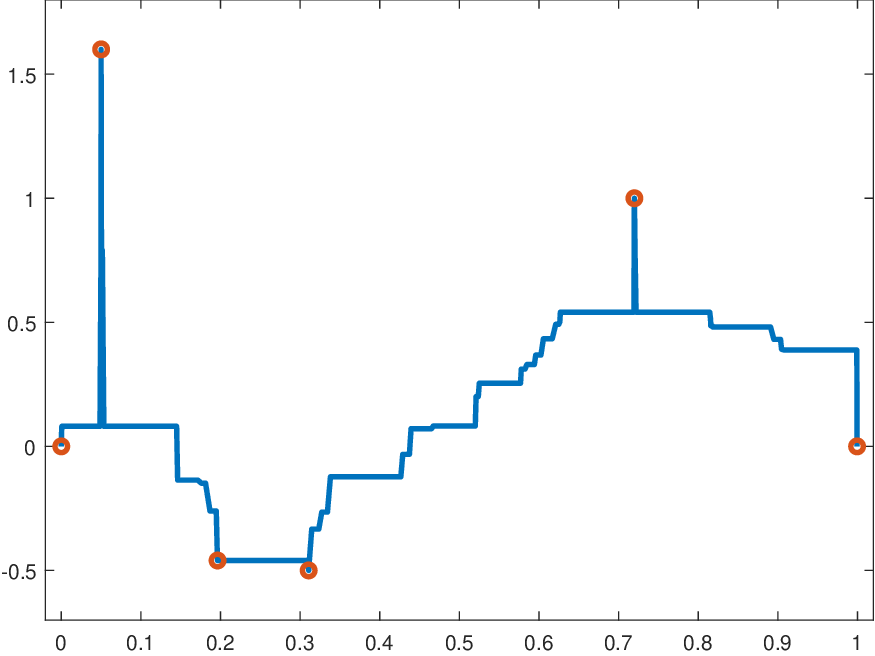} &
        \includegraphics[width=.24\textwidth]{HLp2e0.048.eps} &
          \includegraphics[width=.24\textwidth]{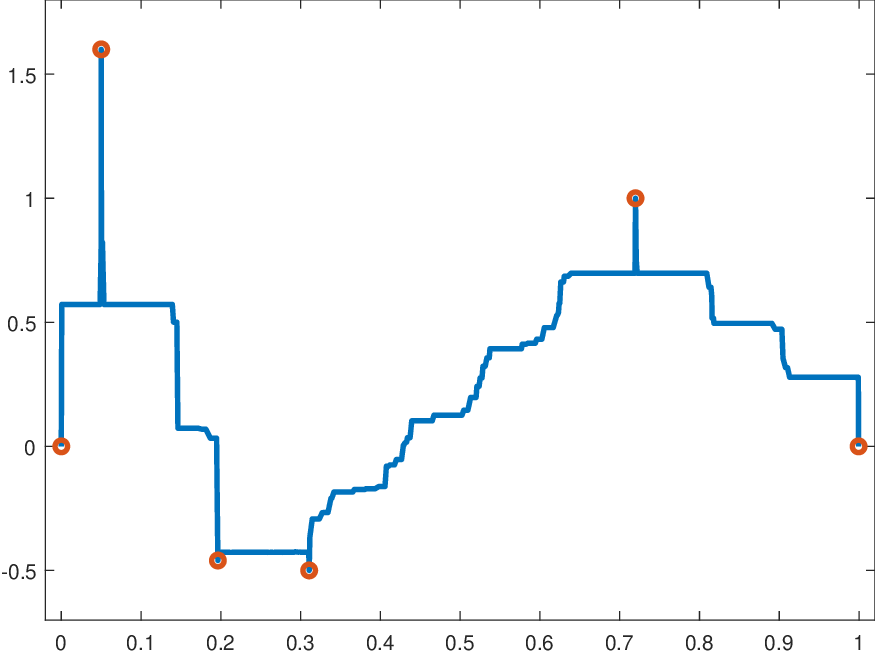} \\
          \includegraphics[width=.24\textwidth]{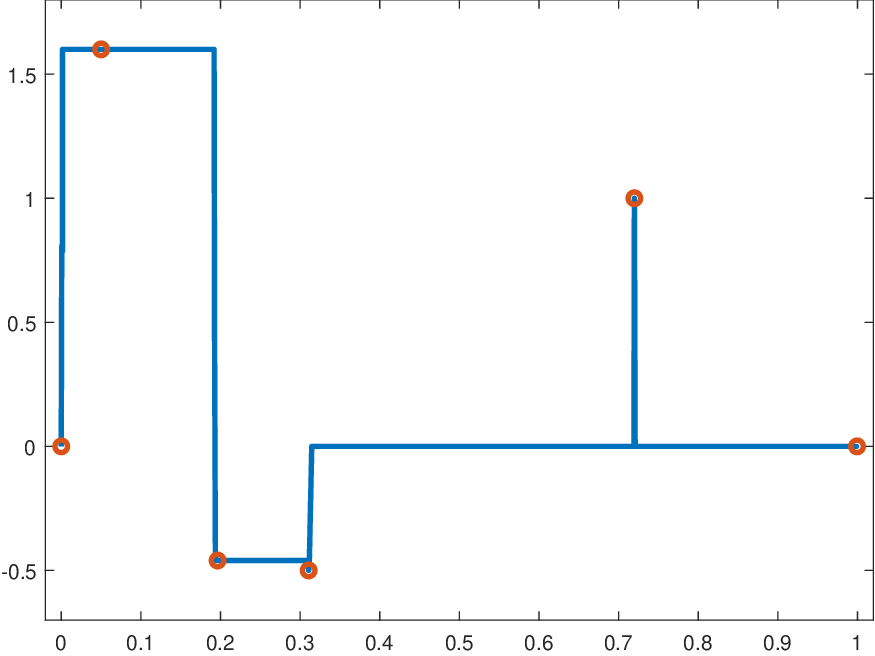} &
          \includegraphics[width=.24\textwidth]{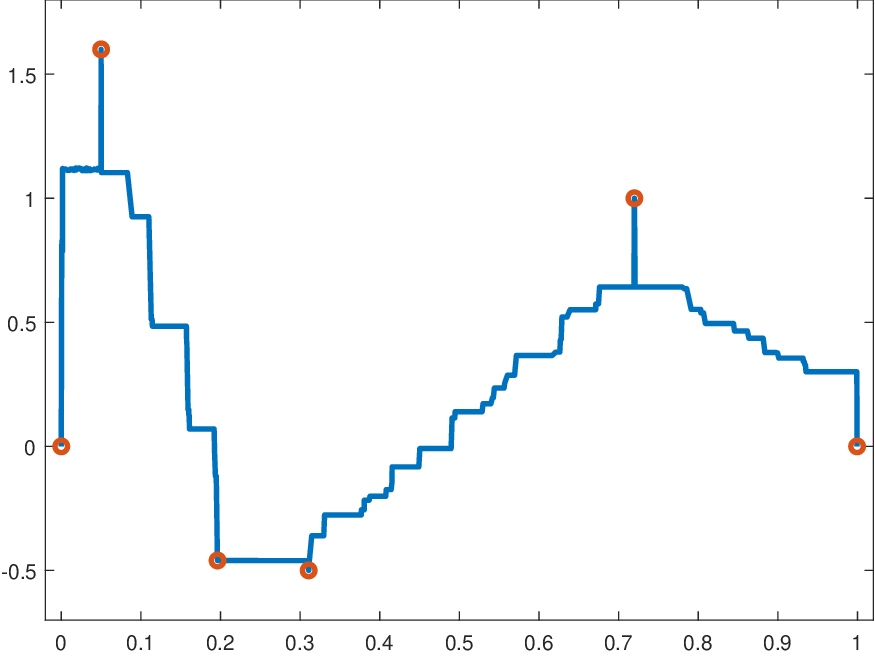} &
          \includegraphics[width=.24\textwidth]{HLp2k72.eps} &
          \includegraphics[width=.24\textwidth]{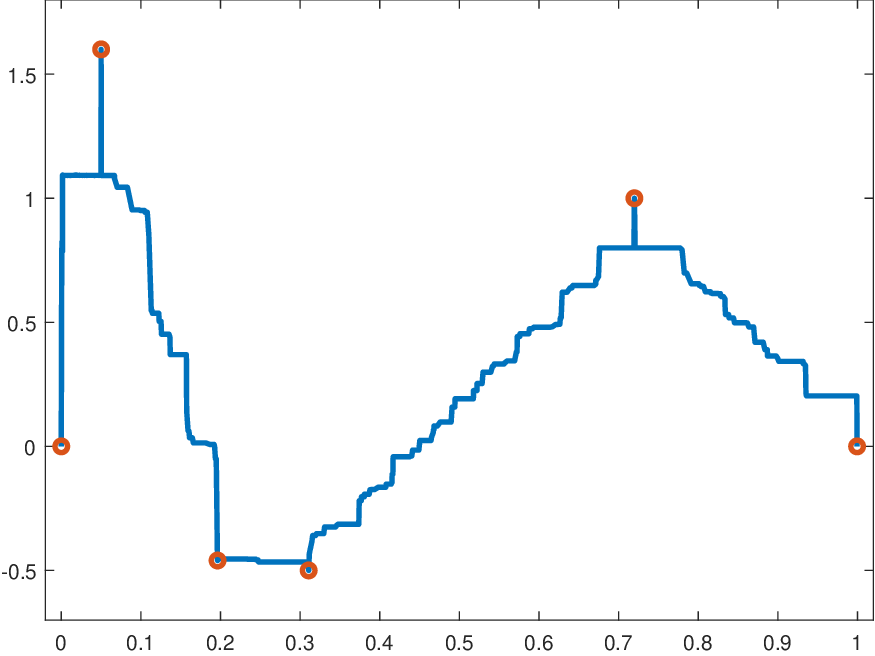}  \\
          \footnotesize\emph{$p=1$}  &
          \footnotesize\emph{$p=1.5$} &
          \footnotesize\emph{$p=2$} &
          \footnotesize\emph{$p=4$} \\
        \end{tabular}
        \caption{Results of HpL with different $p$. First row: $\varepsilon_n$-ball hypergraph with $\varepsilon_n=0.048$; second row: $k_n$-NN hypergraph with $k_n=72$.}
        \label{fig:1Dp}
      \end{figure}

At the end of this subsection, we mention that the $k_n$-NN hypergraph is more desirable in applications.
Let us go back to the first column of Figure \ref{fig:1Deps} and Figure \ref{fig:1Dk}, where the smallest $\varepsilon_n$ and $k_n$ are chosen such that the hypergraph is connected.
The error of HpL on the $k_n$-NN hypergraph is much smaller than the error on the $\varepsilon_n$-ball hypergraph.
Besides, minimizing HpL on the $k_n$-NN hypergraph requires less computational cost.
We shall utilize the $k_n$-NN hypergraph in the following experiments of SSL and image inpainting and let $p=2$ for both GpL and HpL.

\begin{figure}
  \centering
  \includegraphics[width=.4\textwidth]{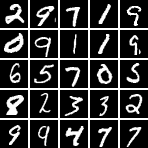}
    \caption{Some images from the MNIST dataset.}
    \label{fig:MNIST}
  \end{figure}

\subsection{Semi-supervised learning}\label{se:SSL}
We test GpL and HpL on the well-known MNIST dataset of handwritten digits \cite{lecun1998gradient}.
It contains 70000 gray-scale images, each of which is $28\times 28$ in size, representing a digit between $0-9$. Figure \ref{fig:MNIST} shows some images in the MNIST dataset.
Following \cite{zhu2003semi}, we regard each image as a point in $\mathbb{R}^{28\times28}$.
The details of SSL with GpL/HpL for the MNIST dataset are summarized in Algorithm \ref{alg3}.

\begin{algorithm}
  \caption{SSL with GpL/HpL for the MNIST dataset.}
  \label{alg3}
  \begin{algorithmic}
  \REQUIRE The MNIST dataset $\Omega_n$, the labeled set $\mathcal{O}$ with labels $u(x_i)=y_i$, $x_i\in\mathcal{O}$, $y_i\in\{0,1,\cdots,9\}$.
  \FOR { $k=0:9$}
  \STATE Find the constraint: For any $x_i\in\mathcal{O}$,
  \begin{align}\label{eq:alg3}
    u_k(x_i)=
    \begin{cases}
      1, & \mbox{if } y_i=k, \\
      0, & \mbox{otherwise}.
    \end{cases}
  \end{align}
  \STATE Solve $u_k$ on $\Omega_n$ by GpL/HpL with constraint \eqref{eq:alg3}.
  \ENDFOR
  \STATE Label points in $\Omega_n\backslash\mathcal{O}$: For any $x_i\in\Omega_n\backslash\mathcal{O}$,
  \begin{equation*}
    l=\arg\max_{0\leq k\leq9}u_k(x_i),\quad u(x_i)=l.
  \end{equation*}
  \RETURN $u$.
  \end{algorithmic}
\end{algorithm}


\begin{table}[htbp]
  \setlength\tabcolsep{3.5pt}
  \centering
  \caption{Mean accuracy (\%) and standard deviation of GpL and HpL for the classification of the MNIST dataset with different labeling rates.}
    \begin{tabular}{llllllll}
    \toprule
     & 0.05\%   & 0.1\% & 0.2\% & 0.5\% & 1\%  & 2\%  & 5\% \\\midrule
    GpL & 16.2$\pm$7.5 & 32.1$\pm$13.3 & 56.6$\pm$13.4 & 86.1$\pm$3.7 & 92.4$\pm$0.9 & 93.9$\pm$0.4 & 95.0$\pm$0.2 \\
    HpL & 50.2$\pm$8.7 & 69.8$\pm$6.2 & 80.7$\pm$5.0 & 87.9$\pm$2.6 & 91.6$\pm$0.7 & 92.9$\pm$0.4 & 94.4$\pm$0.2 \\
    \bottomrule
    \end{tabular}%
  \label{tab:SSL}%
\end{table}%

For both GpL and HpL, we choose $k_n=21$ and the weight function
\begin{equation*}
  w_{i,j}=\exp\left(-\frac{\|x_i-x_j\|^2}{\sigma(x_i)^2}\right),
\end{equation*}
where $\sigma(x_i)$ is the distance between $x_i$ and its the 21th nearest neighbor.
This is calculated by the vlfeat toolbox \cite{vedaldi2010vlfeat}.

Table \ref{tab:SSL} presents the classification accuracy of GpL and HpL with $0.05\%-5\%$ randomly chosen labeled points. We repeat the experiment 20 times and show the mean accuracy and the standard deviation.
When the labeling rate is extremely low ($\leq 0.1\%$), GpL fails to give meaningful results. HpL has a significantly better performance, although it is not satisfactorily accurate.
This is mainly attributed to HpL's ability of better inhibiting spikes.
However, as the labeling rate increases, the gap between the two algorithms becomes smaller.
GpL even has a higher mean accuracy of $1\%$ than HpL in the case of $2\%$ labeling rate.

\begin{algorithm}[t]
  \caption{Image inpainting with GpL/HpL.}
  \label{alg4}
  \begin{algorithmic}
  \REQUIRE Partially observed pixels of an image $f$: $\{f_{i,j},(i,j)\in X\}$; an initial guess $u^{(0)}$.
  \STATE Initialization: $k=0$.
  \FOR { $k=0:K-1$}
  \STATE Construct point cloud $\Omega_n(u^{(k)})$.
  \STATE Construct the training set $\mathcal{O}(u^{(k)})=\{P_{i,j}\in\Omega_n(u^{(k)}), (i,j)\in X\}$ and $u^{(k+1)}(P_{i,j})=f_{i,j}$ for any $P_{i,j}\in \mathcal{O}(u^{(k)})$.
  \STATE Solve $u^{(k+1)}$ on $\Omega_n(u^{(k)})$ by GpL/HpL with the given training set.
  \ENDFOR
  \RETURN $u^{(K-1)}$.
  \end{algorithmic}
\end{algorithm}

\begin{figure}
  \centering
  \begin{tabular}{@{}c@{~}c@{~}c@{~}c@{~}c@{~}c@{}}
    \includegraphics[width=.15\textwidth]{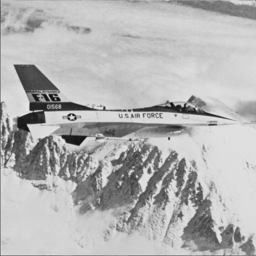} &
    \includegraphics[width=.15\textwidth]{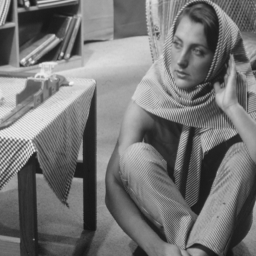} &
    \includegraphics[width=.15\textwidth]{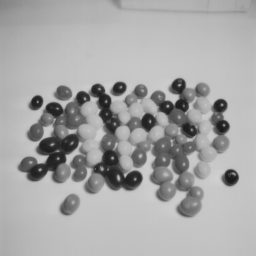} &
    \includegraphics[width=.15\textwidth]{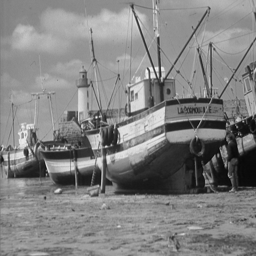} &
    \includegraphics[width=.15\textwidth]{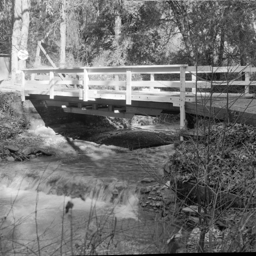} &
    \includegraphics[width=.15\textwidth]{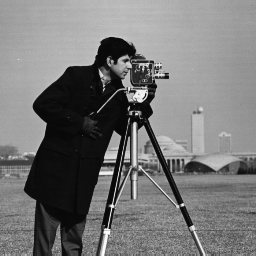} \\
    \footnotesize\emph{Airplane} & \footnotesize\emph{Barbara}  & \footnotesize\emph{Beans}  & \footnotesize\emph{Boat}  & \footnotesize\emph{Bridge}  & \footnotesize\emph{Cameraman} \\
    \includegraphics[width=.15\textwidth]{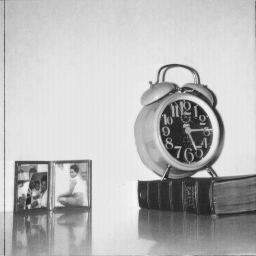} &
    \includegraphics[width=.15\textwidth]{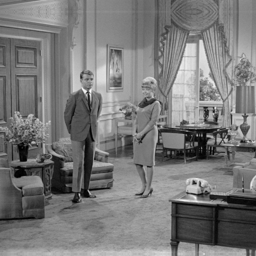} &
    \includegraphics[width=.15\textwidth]{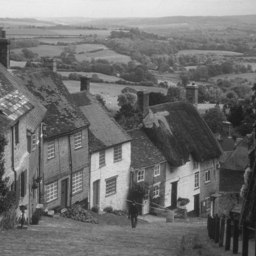} &
    \includegraphics[width=.15\textwidth]{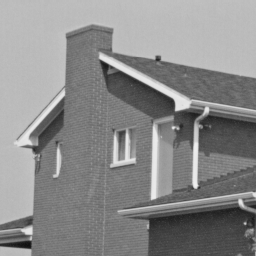} &
    \includegraphics[width=.15\textwidth]{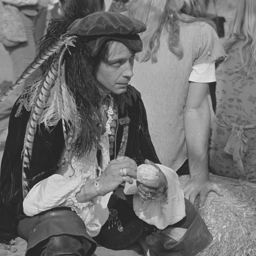} &
    \includegraphics[width=.15\textwidth]{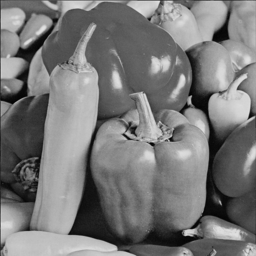} \\
    \footnotesize\emph{Clock} & \footnotesize\emph{Couple} & \footnotesize\emph{Hill} & \footnotesize\emph{House} & \footnotesize\emph{Man} & \footnotesize\emph{Peppers} \\
  \end{tabular}
    \caption{Test images with size $256\times256$ for image inpainting.}
    \label{fig:testimage}
  \end{figure}

\subsection{Image inpainting}
In this subsection, we consider image inpainting from sparse data.
We follow \cite{shi2017weighted} to construct a point cloud from an image.
More precisely, let $f\in \mathbb{R}^{N_1\times N_2}$ be a gray-scale image.
We define image patch $P_{i,j}(f)$ to be a rectangle centered at pixel $(i,j)$ of size $s_1\times s_2$.
A mirror extension of $f$ is needed for pixels that are near the boundary.
Now the point cloud genarated by $f$ is defined as
\begin{equation*}
  \Omega_n(f)=\left\{P_{i,j}(f): 1\leq i\leq N_1, 1\leq j\leq N_2\right\}\subset \mathbb{R}^{s_1\times s_2}.
\end{equation*}
The image patch $P_{i,j}(f)$ and the intensity $f_{i,j}$ of $f$ at pixel $(i,j)$ can be connected via a function $u$ defined on $\Omega_n(f)$, i.e.,
\begin{equation*}
  u(P_{i,j}(f))=f_{i,j}.
\end{equation*}
Then if $\Omega_n(f)$ is known and $\mathcal{O}(f)\subset \Omega_n(f)$, image inpainting with given $u$ on $\mathcal{O}(f)$ becomes a data interpolation problem in $\mathbb{R}^{s_1\times s_2}$.
In practice, the point cloud is unknown. For this reason, we start with an initial guess of $f$ and update the point cloud with the restored image to reduce the error.
The details of image inpainting with GpL and HpL are summarized in Algorithm \ref{alg4}.

\begin{table}
  \setlength\tabcolsep{4.0pt}
  \centering
  \caption{The PSNR and SSIM values of results of GpL and HpL for different test images and different sampling rates.}
    \begin{tabular}{llllllllll}
    \toprule
          &       & \multicolumn{4}{l}{PSNR}      & \multicolumn{4}{l}{SSIM} \bigstrut\\
    \hline
          &       & 5\%   & 10\%  & 15\%  & 20\%  & 5\%   & 10\%  & 15\%  & 20\% \bigstrut\\
    \hline
    Airplane & GpL   & 21.03 & 22.65 & 23.93 & 24.58 & 0.3246 & 0.4516 & 0.5499 & 0.6120 \bigstrut[t]\\
          & HpL   & 21.33 & 23.09 & 24.55 & 25.24 & 0.3825 & 0.5125 & 0.6075 & 0.6618 \\\hline
    Barbara & GpL   & 22.55 & 24.60 & 25.99 & 27.03 & 0.4436 & 0.5767 & 0.6571 & 0.7142 \\
          & HpL   & 23.15 & 25.51 & 27.00 & 27.95 & 0.5043 & 0.6354 & 0.7062 & 0.7553 \\\hline
    Beans & GpL   & 22.66 & 25.12 & 26.35 & 28.26 & 0.4486 & 0.5435 & 0.6097 & 0.6620 \\
          & HpL   & 23.20 & 26.18 & 27.52 & 29.64 & 0.4911 & 0.5877 & 0.6464 & 0.6914 \\\hline
    Boat  & GpL   & 21.96 & 23.14 & 24.18 & 25.21 & 0.3199 & 0.4407 & 0.5408 & 0.6115 \\
          & HpL   & 22.19 & 23.51 & 24.60 & 25.81 & 0.3704 & 0.5000 & 0.5915 & 0.6591 \\\hline
    Bridge & GpL   & 20.31 & 21.69 & 22.50 & 23.24 & 0.2790 & 0.4040 & 0.4874 & 0.5534 \\
          & HpL   & 20.41 & 21.92 & 22.74 & 23.53 & 0.3203 & 0.4496 & 0.5268 & 0.5906 \\\hline
    C.man & GpL   & 20.88 & 22.01 & 23.43 & 24.02 & 0.2257 & 0.3434 & 0.4331 & 0.4943 \\
          & HpL   & 21.07 & 22.31 & 23.85 & 24.53 & 0.2621 & 0.3776 & 0.4688 & 0.5246 \\\hline
    Clock & GpL   & 22.62 & 24.20 & 25.98 & 27.08 & 0.3061 & 0.4256 & 0.5285 & 0.5880 \\
          & HpL   & 22.92 & 24.72 & 26.65 & 27.70 & 0.3374 & 0.4667 & 0.5692 & 0.6227 \\\hline
    Couple & GpL   & 21.55 & 23.35 & 24.32 & 25.38 & 0.2936 & 0.4653 & 0.5486 & 0.6351 \\
          & HpL   & 21.76 & 23.80 & 24.88 & 26.00 & 0.3532 & 0.5302 & 0.6074 & 0.6871 \\\hline
    Hill  & GpL   & 23.43 & 25.28 & 26.33 & 27.06 & 0.3234 & 0.4818 & 0.5671 & 0.6264 \\
          & HpL   & 23.81 & 25.98 & 26.95 & 27.76 & 0.3804 & 0.5425 & 0.6196 & 0.6730 \\\hline
    House & GpL   & 24.55 & 26.99 & 28.76 & 30.19 & 0.2636 & 0.3895 & 0.4616 & 0.5276 \\
          & HpL   & 25.01 & 27.70 & 29.80 & 31.18 & 0.2981 & 0.4217 & 0.4902 & 0.5505 \\\hline
    Man   & GpL   & 22.32 & 24.07 & 25.01 & 25.93 & 0.3156 & 0.4506 & 0.5385 & 0.6034 \\
          & HpL   & 22.75 & 24.66 & 25.71 & 26.61 & 0.3772 & 0.5156 & 0.5971 & 0.6501 \\\hline
    Peppers & GpL   & 22.33 & 24.46 & 25.70 & 27.05 & 0.4696 & 0.5956 & 0.6686 & 0.7263 \\
          & HpL   & 23.05 & 25.50 & 26.75 & 28.14 & 0.5357 & 0.6614 & 0.7242 & 0.7701 \\\hline
    Average  & GpL   & 22.18 & 23.96 & 25.21 & 26.25 & 0.3344 & 0.4640 & 0.5493 & 0.6129 \\
          & HpL   & 22.56 & 24.57 & 25.92 & 27.01 & 0.3844 & 0.5167 & 0.5962 & 0.6530 \bigstrut[b]\\
    \bottomrule
    \end{tabular}%
  \label{tab:PSNR}%
\end{table}%

To calculate the weight function for GpL and HpL, we define the semilocal patch
\begin{equation*}
  \bar{P}_{i,j}(f)=[P_{i,j}(f), \lambda\bar{x}],
\end{equation*}
with
\begin{equation*}
  \bar{x}=\left(\frac{i}{N_1}, \frac{j}{N_2}\right),
\end{equation*}
that includes the local coordinate.
This restricts the search of the $k_n$ nearest neighbors to a local area and accelerates the algorithm.
Then the weight function is calculated as in the case of Section \ref{se:SSL}.
We utilize $s_1=s_2=11$ and $\lambda=10$ for the experiments.
The iteration number $K$ in Algorithm \ref{alg4} is chosen to be $15$ for GpL. While for HpL, we let the result of GpL be the initial guess $u^{(0)}$ and let $K=3$.

The test images are shown in Figure \ref{fig:testimage}.
In Table \ref{tab:PSNR}, we list the PSNR and SSIM values \cite{wang2004image} of the results of GpL and HpL with sampling rates $5\%$, $10\%$, $15\%$, and $20\%$.
HpL outperforms GpL for all test images and all sampling rates.
From Figure \ref{fig:inpainting}, it can be seen that HpL causes fewer spikes than GpL.
This is consistent with previous experimental results.

\begin{figure}[t]
  \centering
  \begin{tabular}{@{~}c@{~}c@{~}c@{~}c@{}}
  \begin{overpic}[width=.23\textwidth]{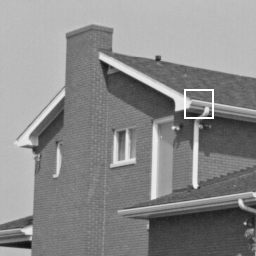}\put(50,0){\includegraphics[width=.115\textwidth]{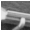}}\end{overpic}
  \begin{overpic}[width=.23\textwidth]{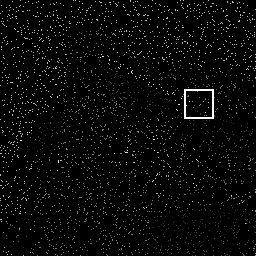}\put(50,0){\includegraphics[width=.115\textwidth]{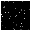}}\end{overpic}
  \begin{overpic}[width=.23\textwidth]{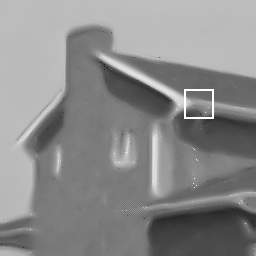}\put(50,0){\includegraphics[width=.115\textwidth]{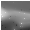}}\end{overpic}
  \begin{overpic}[width=.23\textwidth]{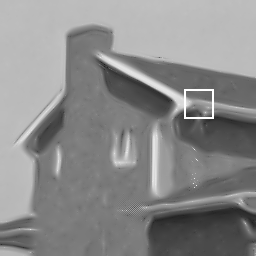}\put(50,0){\includegraphics[width=.115\textwidth]{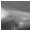}}\end{overpic}\\
  \begin{overpic}[width=.23\textwidth]{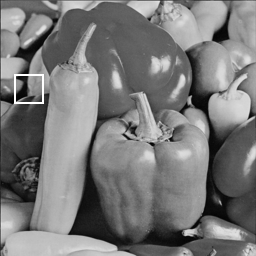}\put(50,0){\includegraphics[width=.115\textwidth]{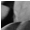}}\end{overpic}
  \begin{overpic}[width=.23\textwidth]{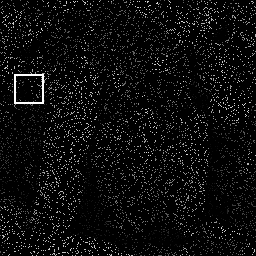}\put(50,0){\includegraphics[width=.115\textwidth]{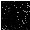}}\end{overpic}
  \begin{overpic}[width=.23\textwidth]{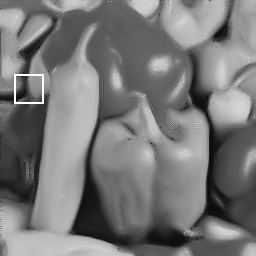}\put(50,0){\includegraphics[width=.115\textwidth]{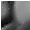}}\end{overpic}
  \begin{overpic}[width=.23\textwidth]{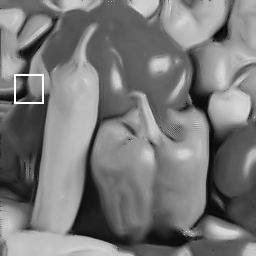}\put(50,0){\includegraphics[width=.115\textwidth]{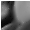}}\end{overpic}\\
  \begin{overpic}[width=.23\textwidth]{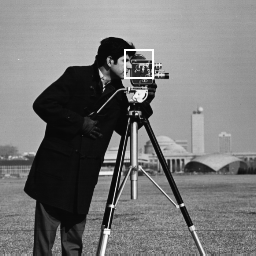}\put(50,0){\includegraphics[width=.115\textwidth]{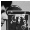}}\end{overpic}
  \begin{overpic}[width=.23\textwidth]{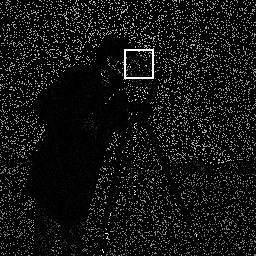}\put(50,0){\includegraphics[width=.115\textwidth]{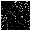}}\end{overpic}
  \begin{overpic}[width=.23\textwidth]{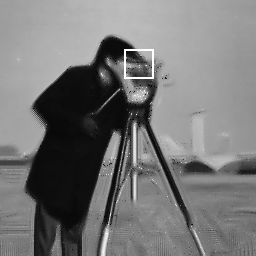}\put(50,0){\includegraphics[width=.115\textwidth]{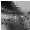}}\end{overpic}
  \begin{overpic}[width=.23\textwidth]{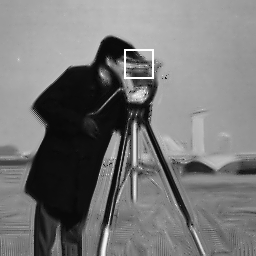}\put(50,0){\includegraphics[width=.115\textwidth]{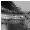}}\end{overpic}\\
  \begin{overpic}[width=.23\textwidth]{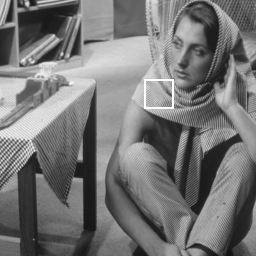}\put(50,0){\includegraphics[width=.115\textwidth]{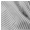}}\end{overpic}
  \begin{overpic}[width=.23\textwidth]{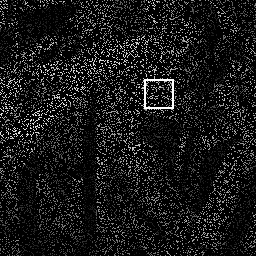}\put(50,0){\includegraphics[width=.115\textwidth]{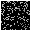}}\end{overpic}
  \begin{overpic}[width=.23\textwidth]{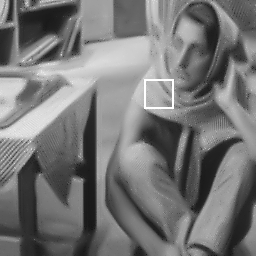}\put(50,0){\includegraphics[width=.115\textwidth]{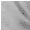}}\end{overpic}
  \begin{overpic}[width=.23\textwidth]{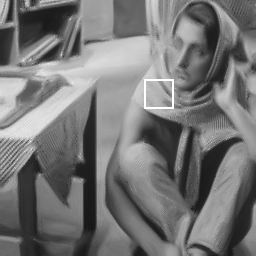}\put(50,0){\includegraphics[width=.115\textwidth]{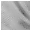}}\end{overpic}\\
  \end{tabular}
    \caption{The restored results of GpL and HpL for four test images.
   From left to right: The test image, The observed pixels, GpL, and HpL.
   From top to bottom (image/sampling rate): House/5\%, Peppers/10\%, Cameraman/15\%, Barbara/20\%.
  }
    \label{fig:inpainting}
  \end{figure}

A main drawback of the new hypergraph $p$-Laplacian regularization is the high computational cost.
In our definition, the hyperedges have the same number as the vertices, resulting in the objective function being a large-scale non-smooth function.
There exists no simple algorithm for solving it, even when $p=2$.
In the experiment of image inpainting, we use the result of GpL as the initial guess for HpL, which saves a lot of computational time.
Nevertheless,
working with MATLAB 2020b on a desktop equipped with an Intel Core i7 3.20 GHz CPU,
the average running time of GpL/HpL is about 165/262 seconds.

\section{Conclusion and future work}
In this paper, we defined the $\varepsilon_n$-ball hypergraph and the $k_n$-nearest neighbor hypergraph from a point cloud and studied the $p$-Laplacian regularization on both hypergraphs in a semisupervised setting.
It was shown that the hypergraph $p$-Laplacian regularization is variational consistent with the continuum  $p$-Laplacian regularization.
Compared to the graph $p$-Laplacian regularization, the new hypergraph regularization was proven both theoretically and numerically to be more effective in suppressing spiky solutions.

Some parts of the numerical results are still not well understood. For example, the staircasing behavior has been observed in experiments in 1D.
The hypergraph model loses its competitiveness in semi-supervised learning with large labeling rates.
Both phenomena are worth further exploring.
Besides, developing new algorithms that require less computational cost is also one of future directions.

 \section*{Acknowledgments}
The authors would like to thank the referee for the valuable comments and suggestions.
KS is supported by China Scholarship Council.
The authors acknowledge support from DESY (Hamburg, Germany), a member of the Helmholtz Association HGF.









%
%

\bibliographystyle{unsrt}
\bibliography{references}

\end{document}